\documentclass[a4paper,oneside,11pt]{amsart}

\usepackage[utf8]{inputenc}
\usepackage[T1]{fontenc}
\usepackage[english]{babel}

\usepackage{geometry}
\usepackage{needspace}

\usepackage{amssymb, amsfonts}

\usepackage{bbm}

\usepackage[all]{xy}

\usepackage{enumerate}

\usepackage{float}
\usepackage{booktabs}
\usepackage{multirow}

\numberwithin{equation}{section}

\newtheorem{theoremcounter}{theoremcounter}[section]

\newtheorem{corollary}[theoremcounter]{Corollary}
\newtheorem{definition}[theoremcounter]{Definition}
\newtheorem{example}[theoremcounter]{Example}
\newtheorem{lemma}[theoremcounter]{Lemma}
\newtheorem{proposition}[theoremcounter]{Proposition}
\newtheorem{remark}[theoremcounter]{Remark}

\newtheorem{theorem}[theoremcounter]{Theorem}

\newcommand{\tit}{\itshape}

\newcommand{\nbd}{\nobreakdash-\hspace{0pt}}

\newcommand{\ZZ}{\ensuremath{\mathbb{Z}}}
\newcommand{\QQ}{\ensuremath{\mathbb{Q}}}
\newcommand{\RR}{\ensuremath{\mathbb{R}}}
\newcommand{\CC}{\ensuremath{\mathbb{C}}}

\newcommand{\sgn}{\ensuremath{\mathrm{sgn}}}

\newcommand{\SL}[1]{\ensuremath{\mathrm{SL}_{#1}}}

\newcommand{\slashdiv}{\ensuremath{\mathop{/}}}

\newcommand{\HS}{\mathbb{H}}

\renewcommand{\frak}{\ensuremath{\mathfrak}}
\newcommand{\cal}{\ensuremath{\mathcal}}
\newcommand{\bboard}{\ensuremath{\mathbb}}

\newcommand{\frake}{\ensuremath{\frak{e}}}

\newcommand{\cC}{\ensuremath{\cal{C}}}

\newcommand{\cE}{\ensuremath{\cal{E}}}
\newcommand{\cF}{\ensuremath{\cal{F}}}

\newcommand{\cM}{\ensuremath{\cal{M}}}
\newcommand{\cN}{\ensuremath{\cal{N}}}

\newcommand{\cP}{\ensuremath{\cal{P}}}

\newcommand{\bbM}{\ensuremath{\bboard M}}

\newcommand{\rmJ}{\ensuremath{\mathrm{J}}}

\newcommand{\rmh}{\ensuremath{\mathrm{h}}}

\newcommand{\td}{\tilde}

\newcommand{\ov}{\overline}

\newcommand{\Z}{\ensuremath{\mathbb{Z}}}
\newcommand{\R}{\ensuremath{\mathbb{R}}}
\newcommand{\C}{\ensuremath{\mathbb{C}}}
\newcommand{\HH}{\ensuremath{\mathbb{H}}}

\newcommand{\Mp}[1]{\ensuremath{\mathrm{Mp}_{#1}}}

\newcommand{\JacF}{\ensuremath{\Gamma^{\rm J}}}
\newcommand{\sk}{{\rm sk}}

\newcommand{\Hone}{{\rm H}}

\newcommand{\holJ}{\rmJ^{\delta, \rmh}}
\newcommand{\holJJ}{\rmJ^{\Delta, \rmh}}
\newcommand{\holJH}{\rmJ^{\delta, \Hone}}
\newcommand{\holJJH}{\rmJ^{\Delta, \Hone}}

\newcommand{\holJkm}{\holJ_{k,m}}
\newcommand{\holJJkm}{\holJJ_{k,m}}
\newcommand{\holJHkm}{\holJH_{k,m}}
\newcommand{\holJJHkm}{\holJJH_{k,m}}

\newcommand{\J}{\cM \holJ}
\newcommand{\JJ}{\cM \holJJ}
\newcommand{\JH}{\cM \holJH}
\newcommand{\JJH}{\cM \holJJH}

\newcommand{\Jkm}{\J_{k,m}}
\newcommand{\JJkm}{\JJ_{k,m}}
\newcommand{\JHkm}{\JH_{k,m}}
\newcommand{\JJHkm}{\JJH_{k,m}}

\newcommand{\holJsk}{\rmJ^{\sk, \delta, \rmh}}
\newcommand{\holJJsk}{\rmJ^{\sk, \Delta, \rmh}}
\newcommand{\holJHsk}{\rmJ^{\sk, \delta, \Hone}}
\newcommand{\holJJHsk}{\rmJ^{\sk, \Delta, \Hone}}

\newcommand{\holJskkm}{\holJsk_{k,m}}
\newcommand{\holJJskkm}{\holJJsk_{k,m}}

\newcommand{\Jsk}{\cM \holJsk}
\newcommand{\JJsk}{\cM \holJJsk}
\newcommand{\JHsk}{\cM \holJHsk}
\newcommand{\JJHsk}{\cM \holJJHsk}

\newcommand{\Jskkm}{\Jsk_{k,m}}
\newcommand{\JJskkm}{\JJsk_{k,m}}
\newcommand{\JHskkm}{\JHsk_{k,m}}
\newcommand{\JJHskkm}{\JJHsk_{k,m}}

\newcommand{\lapH}{\Delta^{\rm H}}

\newcommand{\xiJ}{\xi}
\newcommand{\xiJsk}{\xi^{\sk}}
\newcommand{\xiJH}{\xi^{\rm H}}
\newcommand{\xiJHsk}{\xi^{\sk,{\rm H}}}

\newcommand{\xiJkm}{\xi_{k,m}}
\newcommand{\xiJskkm}{\xi^{\sk}_{k,m}}
\newcommand{\xiJHkm}{\xi^{\rm H}_{k,m}}
\newcommand{\xiJHskkm}{\xi^{\sk, {\rm H}}_{k,m}}

\newcommand{\Hharmonic}{{${\rm H}$\nbd harmonic}}
\newcommand{\Hquasi}{{${\rm H}$\nbd quasi}}

\renewcommand{\pmod}[1]{\;({\rm mod}\,{#1})}

\begin{document}

\title[Harmonic Maass-Jacobi forms]{Harmonic Maass-Jacobi forms with singularities\\ and a theta-like decomposition}

\author{Kathrin Bringmann}
\address{Mathematisches Institut, Universit\"at zu K\"oln\\ Weyertal 86-90, D-50931 K\"oln\\Germany}
\email{kbringma@math.uni-koeln.de} 

\author{Martin Raum}
\address{ETH Zurich, Dept. Mathematics, R\"amistrasse~101, CH-8092 Z\"urich, Switzerland}
\email{martin@raum-brothers.eu} 
\urladdr{http://www.raum-brothers.eu/martin}

\author{Olav K. Richter}
\address{Department of Mathematics\\University of North Texas\\ Denton, TX 76203\\USA}
\email{richter@unt.edu}

\thanks{The first author was partially supported by the Alfried Krupp Prize for Young University Teachers of the Krupp Foundation and by NSF grant DMS-$0757907$. The second author held a scholarship from the Max Planck society, and is supported by the ETH Zurich Postdoctoral Fellowship Program and by the Marie Curie Actions for People COFUND Program.  The third author was partially supported by Simons Foundation Grant $\#200765$}




\begin{abstract}

Real-analytic Jacobi forms play key roles in different areas of mathematics and physics, but a satisfactory theory of such Jacobi forms has been lacking.  In this paper, we fill this gap by introducing a space of harmonic Maass-Jacobi forms with singularities which includes the real-analytic Jacobi forms from Zwegers's PhD thesis.  We provide several structure results for the space of such Jacobi forms, and we employ Zwegers's $\widehat{\mu}$-functions to establish a theta-like decomposition.
\end{abstract}

\maketitle

\section{Introduction}

Jacobi forms have a long history, and they provide deep links between different types of automorphic objects.  An extraordinary Jacobi form is Zwegers's real-analytic function $\widehat{\mu}$, which is a crucial tool in his PhD thesis~\cite{Zwe-thesis} on mock theta functions.  This $\widehat{\mu}$\nbd function and similar real-analytic Jacobi forms also play a role in the theory of Donaldson invariants of $\C\mathbb{P}^2$ that are related to gauge theory (see for example G{\"o}ttsche and Zagier \cite{GZ-Selecta98},  G{\"o}ttsche, Nakajima, Yoshioka~\cite{GNY-DiffGeom08}, and Malmendier and Ono~\cite{Ono-Mal-GeomTop12}), and they also appear in the Mathieu moonshine (see for example Eguchi, Ooguri, and Tachikawa \cite{Egu-Oog-Tac}).  Naturally, one wishes to better understand real-analytic Jacobi forms.  In~\cite{B-R-Maass-Jacobi}, the first and third author suggest a definition of harmonic Maass-Jacobi forms, which up to singularities includes Zwegers's $\widehat{\mu}$\nbd function.  However, the definition in~\cite{B-R-Maass-Jacobi} only allows Jacobi forms without singularities, and hence the $\widehat{\mu}$-function itself does not belong to the space of such forms.  Another drawback is that the entire space of Jacobi forms in \cite{B-R-Maass-Jacobi} is too large, and it seems impossible to analyze the structure of that space as a whole.

In this paper, we improve the definition in~\cite{B-R-Maass-Jacobi} by introducing the space $\JJHkm$ of Heisenberg~harmonic (\Hharmonic) Maass-Jacobi forms (see Definition~\ref{def:maassjacobiforms}).  This is a space of real-analytic Jacobi forms with singularities that are annihilated by the Casimir operator~$\cC_{k,m}$ in~\eqref{Casimir} and also by the Heisenberg Laplace operator~$\lapH_m$ (a Jacobi form analogue of the usual Laplace operator) in~\eqref{eq:heisenberg-laplace}. This new space of Jacobi forms contains Zwegers's $\widehat{\mu}$-function.  We are able to describe this space explicitly, and we give a series of structure results for it.  We now explain our main results in more detail.

Recall that the Fourier series expansion of a harmonic weak Maass forms consists of a holomorphic part and a non-holomorphic part.  The holomorphic part has the shape of a weakly-holomorphic modular form, while the non-holomorphic part is more complicated and also features the special function $H$ in~\eqref{eq:H}. Bruinier and Funke's \cite{B-F-Duke04} operator $\xi_k$ maps harmonic weak Maass forms of weight $k$ to weakly-holomorphic modular forms of weight $2-k$.  Hence, one may view $\xi_k$ as a differential operator that ``simplifies'' the space of harmonic weak Maass forms.  We encounter similar phenomena in our situation.  We consider the differential operators $\xiJHkm$ (defined in \eqref{eq:xiJH}) and $\xiJ_{k,m}$ (defined in \eqref{eq:xiJ}), which are analogs of $\xi_k$.  \Hharmonic\ Maass-Jacobi forms that are annihilated by these operators are Jacobi forms with an easier structure.  For example, if a \Hharmonic\ Maass-Jacobi form without singularities is annihilated by $\xiJHkm$, then it is {\it semi-holomorphic}, i.e., holomorphic in the Jacobi variable $z$.

We introduce the following spaces of Jacobi forms of weight $k$ and index $m$, where here and throughout the paper we always assume that $k$ and $m$ are half-integers and that $m\not= 0$: The subspaces of forms in $\JJHkm$ that are annihilated by $\xiJ_{k,m}$ and $\xiJHkm$ are denoted by $\JHkm$ and $\JJkm$, respectively, and $\Jkm:=\JHkm\cap\JJkm$.  We write $\holJkm \subseteq \Jkm$, $\holJJkm \subseteq \JJkm$, $\holJHkm\subseteq\JHkm$, and $\holJJHkm\subseteq\JJHkm$ for the subspaces of forms without singularities.  Note that we suppress the superscript $!$ that some authors would use to distinguish the space of holomorphic and weakly holomorphic Jacobi forms.  Table~\ref{tab:maass_jacobi_spaces} lists key characteristics of the above spaces. The first four spaces consist of smooth functions, while the last four spaces include Jacobi forms with singularities.  The prefix ``$\cM$'' stands for ``meromorphic singularities''; see Corollary~\ref{cor:singularities_of_maass_jacobi_forms} in this context.  
\begin{table}[H]
\framebox[30em][c]{
\begin{tabular}{r@{\hspace{3ex}}ccc}
\ \\
\toprule
         & \multicolumn{3}{c}{Elements are} \\
         \cline{2-4}\noalign{\smallskip}
         & smooth                    & annihilated by & annihilated by \\
  Space  & \hphantom{annihilated by} & $\xiJHkm$      & $\xiJkm$ \\[3pt]
\midrule
$\holJkm$  & \checkmark & \checkmark & \checkmark \\[3.5pt]
$\holJJkm$ & \checkmark & \checkmark & ---  \\[3.5pt]
$\holJHkm$  & \checkmark & --- & \checkmark \\[3.5pt]
$\holJJHkm$ & \checkmark & --- & --- \\[3.5pt]
$\Jkm$ & --- & \checkmark & \checkmark \\[3.5pt]
$\JJkm$ & --- & \checkmark & --- \\[3.5pt]
$\JHkm$ & --- & --- & \checkmark \\[3.5pt]
$\JJHkm$ & --- & --- & --- \\[2pt]
\bottomrule
\\
\end{tabular}
}
\vspace{1ex}
\caption{Spaces of ${\rm H}$-harmonic Maass-Jacobi forms}
\label{tab:maass_jacobi_spaces}
\end{table}

\vspace{1ex}

In Sections~\ref{sec:harmonicjacobiforms} and~\ref{sec:singluarjacobiforms} we also study skew-Maass-Jacobi forms, but here we only give structure results for the spaces in Table~\ref{tab:maass_jacobi_spaces}.

\begin{theorem}
\label{thm:structure_maass_jacobi_forms}
\ 
\begin{enumerate}[(1)]
\item
\label{thm:it:singularity_free_hharmonic}
We have $\holJJHkm = \holJJkm$ and $\holJHkm = \holJkm$, i.e., any \Hharmonic\ Maass-Jacobi form without singularities is semi-holomorphic.  If $m < 0$, then $\holJJkm = \holJkm = \{ 0 \}$.
\item 
\label{thm:it:semimeromorphic_maass_jacobi_forms}
We have $\JJ_{k,m}= \holJJ_{k,m} + \J_{k,m}$, i.e., any $\phi \in \JJkm$ can be written as the sum of a semi-holomorphic Jacobi form and a meromorphic Jacobi form. In particular, if $0\not=\phi \in \JJkm$ is not meromorphic, then $m > 0$.
\item
\label{thm:it:decomposition_of_maass_jacobi_forms}
We have $\JJH_{k,m} = \holJJ_{k,m} + \JH_{k,m}$, i.e., any $\phi \in \JJHkm$ can be written as the sum of a semi-holomorphic Jacobi form and a Jacobi form that is annihilated by $\xiJkm$. In particular, if $0\not=\phi \in \JJHkm$ does not vanish under $\xiJkm$, then $m > 0$.
\item
\label{thm:it:positive_index_implies_semi_meromorphic}
If $m > 0$, then $\JJHkm = \JJkm$ and $\JHkm = \Jkm$, i.e., every \Hharmonic\ Maass-Jacobi form of positive index is semi-holomorphic.
\end{enumerate} 
\end{theorem}





\Needspace*{5em}
Before we continue, we give examples of the spaces given in Table~\ref{tab:maass_jacobi_spaces}.
\begin{example}
\label{ex:maass_jacobi_spaces}
\ 
\begin{enumerate}[(1)]
\item The usual Jacobi forms and weak Jacobi forms in \cite{EZ} belong to $\holJ_{k, m}$ (which is $\holJHkm$ by Theorem~\ref{thm:structure_maass_jacobi_forms}~(\ref{thm:it:singularity_free_hharmonic})).

\item The semi-holomorphic Jacobi-Poincar\'e series~$\cP_{k,m}^{(n,r)}$ in~\cite{B-R-Maass-Jacobi} are examples of $\holJJkm$ (which is $\holJJHkm$ by Theorem~\ref{thm:structure_maass_jacobi_forms}~(\ref{thm:it:singularity_free_hharmonic})).

\item If $0\not=\phi\in\holJkm$, then $\frac{1}{\phi}\in\J_{-k,-m}$.  

\item Theorem~\ref{thm:structure_maass_jacobi_forms}~(\ref{thm:it:semimeromorphic_maass_jacobi_forms}) asserts that a typical element in $\JJkm$ is a sum of a semi-holomorphic Maass-Jacobi form and a meromorphic Jacobi form.  For example, let $\cP_{k,m}^{(n,r)}$ be the semi-holomorphic Jacobi-Poincar\'e series in~\cite{B-R-Maass-Jacobi}, and let $\phi_{10,1}$ and $\phi_{12,1}$ be the usual Jacobi cusp forms of index $1$ and weights $10$ and $12$, respectively. Then $\cP_{14, 1}^{(12, 0)} + \frac{( \phi_{12,1})^2}{\phi_{10,1}}\in\JJ_{14, 1}$.

\item Let $\phi \in \holJskkm$ (defined in Section~\ref{sec:harmonicjacobiforms}) with theta decomposition $\phi= \sum_{l \pmod{2 m}} h_l\, \theta_{m,l}$ and $\widehat{\mu}_{m,l}$ as in (\ref{eq:mu_hat_ml_definition}). Theorem~\ref{thm:hharmonic_maass_jacobi_forms_thetadecomposition} implies that $\sum_{l \pmod{2 m}} h_l \, \widehat{\mu}_{m,l}\in\JH_{k,-m}$.

\item
Theorem~\ref{thm:structure_maass_jacobi_forms}~(\ref{thm:it:decomposition_of_maass_jacobi_forms}) gives $\JJH_{k,-m} = \JJ_{k,-m} + \JH_{k,-m}$, which shows how to construct examples of forms in $\JJH_{k,-m}$.

\item
\label{ex:maass_jacobi_spaces:it:mu_function}
Zwegers's \cite{Zwe-thesis} real-analytic Jacobi form $\widehat{\mu}$ has a decomposition of the form $\widehat{\mu} = \mu_1 + \widehat{\mu}_2$, where $\mu_1$ is a meromorphic Jacobi form on $\HH\times\C^2$ and where $\widehat{\mu}_2$ is a real analytic Jacobi form on $\HH\times\C$ (see the footnote (1) on page 7 of \cite{Z-Bourbaki} and also \eqref{eq:mu_onehalf_def}).  Note that the image of $\widehat{\mu}_2$ under $\xi_{\frac12,-\frac12}$ was given incorrectly in \cite{B-R-Maass-Jacobi}, and it should have been $\xi_{\frac12,-\frac12}(\widehat{\mu}_2)=0$. One finds that $\widehat{\mu}_2 \in \JH_{\frac{1}{2}, -\frac{1}{2}}$.
\end{enumerate}
\end{example}

Recall that harmonic weak Maass forms are uniquely determined by their singularities at the cusps up to holomorphic modular forms, which are zero for negative weight. Theorem~\ref{thm:structure_maass_jacobi_forms}~(\ref{thm:it:singularity_free_hharmonic}) provides the corresponding result for \Hharmonic\ Maass-Jacobi forms. Specifically, \Hharmonic\ Maass-Jacobi forms are uniquely determined by their singularities up to semi-holomorphic Maass-Jacobi forms, which are zero for negative index.  Note that the recent work of Dabholkar, Murthy, and Zagier \cite{DMZ} on quantum black holes and mock modular forms features {\it mock Jacobi forms}, which may be viewed as the holomorphic parts of semi-holomorphic Maass-Jacobi forms.  Theorem~\ref{thm:structure_maass_jacobi_forms}~(\ref{thm:it:singularity_free_hharmonic}) implies that \cite{DMZ} investigates precisely the holomorphic parts of \Hharmonic\ Maass-Jacobi forms without singularities. Such Jacobi forms play also an important role in fully understanding modularity properties of Kac-Wakimoto characters (see the first author and Olivetto \cite{B-Oliv}).



%

We now turn our attention to another main result.  The classical Jacobi forms in Eichler and Zagier \cite{EZ} have a {\it theta decomposition}, which can be phrased as in \eqref{eq:theta_expansion_map}.  It is easy to see that the semi-holomorphic Maass-Jacobi forms in \cite{B-R-Maass-Jacobi} also have such a theta decomposition.  In this paper, we employ the $\widehat{\mu}$\nbd functions from Zwegers \cite{Zwe-thesis, Zwe-multivar-Appell} to establish a {\it theta-like decomposition} for \Hharmonic\ Maass-Jacobi forms.  More precisely, let $M^!_{k - \smash{\frac{1}{2}},\check \rho_m}$ be the space of weakly holomorphic vector-valued modular forms of weight $k-\frac{1}{2}$ and type $\check \rho_m$ (see Section~\ref{sec:thetadecompositions} for details), and let $\widehat{\mu}_{m,l}$ be the (completed) vector-valued Jacobi form defined in (\ref{eq:mu_hat_ml_definition}), and which is a specialization of Zwegers's function in \cite{Zwe-multivar-Appell}. Theorem~\ref{thm:hharmonic_maass_jacobi_forms_thetadecomposition} gives the theta-like decomposition for \Hharmonic\ Maass-Jacobi forms,  which can also be stated as follows:
\begin{theorem}
\label{thm:intro_thetadecomposition}
Let $m > 0$. The spaces $M^!_{k - \frac{1}{2}, \check \rho_m}$ and $\JH_{k,-m} \slashdiv \J_{k,-m}$ are isomorphic via
\begin{gather*}
  \big( h_l \big)_l
\longmapsto
  \sum_{l \pmod{2 m}} h_l \, \widehat{\mu}_{m,l}
\text{.}
\end{gather*}
\end{theorem}



The theta decomposition of classical Jacobi forms in \cite{EZ} has a natural explanation in terms of representation theory, which is discussed in detail in Berndt and Schmidt~\cite{BS}.  Specifically, let $\pi_{\rm SW}^m$ be the Schr\"odinger-Weil representation of the real Jacobi group with a certain central character.  Then the map
$$
\widetilde{\pi}\longmapsto\pi:=\widetilde{\pi} \otimes \pi_{\rm SW}^m
$$
gives a one-to-one correspondence between genuine automorphic representations $\widetilde{\pi}$ of the metaplectic double cover of $\SL{2}(\RR)$ and automorphic representations $\pi$ of the real Jacobi group.  It would be interesting to find such a representation theoretic interpretation of the theta-like decomposition in Theorem~\ref{thm:intro_thetadecomposition}.  Note that there is no such immediate analog in representation theoretic language, since nontrivial elements of $\JH_{k,-m} \slashdiv \J_{k,-m}$ correspond to functions that are not in $L^2$.  We expect that a combination of Theorems~\ref{thm:structure_maass_jacobi_forms} and~\ref{thm:intro_thetadecomposition} will yield new relations of certain quantities that, so far, have been treated by means of mixed mock modular forms (for example, ``contributions of bounded states of two primitive constitutents with primitive $D4$-brane charges to the full $\cN = 2$ supergravity partition function''; see Section~4 and Appendix~A in Manschot~\cite{Man-ComMatPhy10}).

The paper is organized as follows.  In Section~\ref{sec:differentialoperators}, we review differential operators for the Jacobi group.  In Section~\ref{sec:harmonicjacobiforms}, we define \Hharmonic\ Maass-Jacobi forms, and we explore their Fourier series expansions.  In Section~\ref{sec:singluarjacobiforms}, we apply tools from complex analysis of several variables to study Maass-Jacobi forms with singularities, and we prove Theorem~\ref{thm:structure_maass_jacobi_forms}~(\ref{thm:it:semimeromorphic_maass_jacobi_forms}).  In Section~\ref{sec:thetadecompositions}, we determine the structure of \Hharmonic\ Maass-Jacobi forms, and we prove Theorem~\ref{thm:structure_maass_jacobi_forms}~(\ref{thm:it:singularity_free_hharmonic}), (\ref{thm:it:decomposition_of_maass_jacobi_forms}), and~(\ref{thm:it:positive_index_implies_semi_meromorphic}), and Theorem~\ref{thm:intro_thetadecomposition}. Finally, in Section~\ref{sec:quasijacobiforms} we discuss \Hquasi\ Maass-Jacobi forms, which are real-analytic Jacobi form analogs of the usual quasimodular forms.

\section{Differential operators for the Jacobi group}
\label{sec:differentialoperators}

In this section, we briefly review differential operators for the Jacobi group.  Such operators have been studied in detail in the integral weight case (see \cite{BS, Ameya, Raum-thesis}), but it is easy to see that the results carry over to the half-integral weight setting.  We will summarize these results, after introducing necessary notation.  Throughout, we write $\tau=x+iy\in\HH$ (the usual complex upper half plane) and  $z=u+iv\in\C$.  Recall that the metaplectic cover $\Mp{2}(\RR)$ of $\SL{2}(\RR)$ is the group of pairs $(g, \omega)$, where $g = \left(\begin{smallmatrix}a & b \\ c & d\end{smallmatrix}\right)\in \SL{2}(\RR)$ and $\omega \,:\, \HS \rightarrow \C,\, \tau \mapsto \sqrt{c \tau + d}$ for a holomorphic choice of the square root, with group law
\begin{gather*}
  (g, \omega) (g', \omega')
=
  (g g', \omega \circ g' \cdot \omega')
\text{.}
\end{gather*}
Let $G^{\rm J} := G^{\rm J}(\R) : = \Mp{2}(\R) \ltimes (\R^2 \td\times \R)$ be the {\em metaplectic real Jacobi group} with group law
$$ \bigl(M,X,\kappa\bigr) \bigl(M',X',\kappa'\bigr) := 
   \bigl(MM', XM'+X', \det{ \genfrac(){0pt}{}{XM'}{X'}} 
   + \kappa + \kappa'\bigr)
$$
and let $\Gamma^{\rm J}:=\Mp{2}(\Z)\ltimes \Z^2$ be the full Jacobi group, where $\Mp{2}(\Z)$ is the preimage of $\SL{2}(\Z)$ in $\Mp{2}(\R)$. 
For fixed half-integers $k$ and $m$, and for all
$A=\left[\big(\left(\begin{smallmatrix}a & b\\c & d\end{smallmatrix}\right), \, \sqrt{c \tau + d} \,\big), (\lambda, \mu), \kappa\right]\in G^J$, define the following slash operators on functions $\phi:\HH \times\C\rightarrow\C$\,: 
\begin{align}
\label{Jacobi-slash}
&
  \Big(\phi\,\big|_{k,m} A \Big)(\tau,z)
:=
\\[4pt]
\nonumber 
&
\quad
  \phi \Big(\frac{a\tau+b}{c\tau+d},\frac{z+\lambda\tau+\mu}{c\tau+d} \Big) \,
  (\sqrt{c\tau+d})^{-2k}\, e^{2\pi im\bigl(-\frac{c(z+\lambda\tau+\mu)^2}{c\tau+d}+\lambda^2\tau+2\lambda z+\lambda\mu+\kappa\bigr)}
\end{align}
and
\begin{align}
\label{skew-slash}
&
  \Big(\phi\,\big|_{k,m}^{sk} A \Big)(\tau,z)
:=
\\[4pt]
\nonumber
&
\quad
  \phi \Big( \frac{a\tau+b}{c\tau+d}, \frac{z+\lambda\tau+\mu}{c\tau+d} \Big) \,
  (\sqrt{c\overline{\tau}+d})^{2-2k}\,|c\tau+d|^{-1} \,
  e^{2\pi im\bigl(-\frac{c(z+\lambda\tau+\mu)^2}{c\tau+d}+\lambda^2\tau+2\lambda z+\lambda\mu+\kappa\bigr)}
\text{.}
\end{align}
\vspace{1ex}

If $\kappa=0$, then by a slight abuse of notation we write $\left[\big(\left(\begin{smallmatrix}a & b\\c & d\end{smallmatrix}\right), \sqrt{c \tau + d} \,\big), (\lambda, \mu)\right]$ instead of 
$\left[\big(\left(\begin{smallmatrix}a & b\\c & d\end{smallmatrix}\right), \sqrt{c \tau + d} \,\big), (\lambda, \mu), 0\right]\in G^J$.  For convenience, we define
\begin{align*}
  \partial_{\tau}
&:=
  \frac{\partial}{\partial \tau}
=
  \frac12\left(\frac{\partial}{\partial x}-i\frac{\partial}{\partial y}\right)
\text{,}
&
  \partial_{\overline{\tau}}
&:=
  \frac{\partial}{\partial \overline{\tau}}
=
  \frac12\left(\frac{\partial}{\partial x}+i\frac{\partial}{\partial y}\right)
\text{,}
\\[6pt]
  \partial_{z}
&:=
  \frac{\partial}{\partial z}
=
  \frac12\left(\frac{\partial}{\partial u}-i\frac{\partial}{\partial v}\right)
\text{,}
&
  \partial_{\overline{z}}
&:=
  \frac{\partial}{\partial \overline{z}}
=
  \frac12\left(\frac{\partial}{\partial u}+i\frac{\partial}{\partial v}\right)
\text{.}
\end{align*}
The {\em raising operators} and {\em lowering operators} with respect to actions in (\ref{Jacobi-slash}) and (\ref{skew-slash}) are given by the differential operators,

\begin{align*}
   X^{k,m}_+ &:= 2i \Big(\partial_\tau + \frac{v}{y}\partial_z + 2\pi i m \frac{v^2}{y^2}\Big)
   + \frac{k}{y}
\text{,}
&
   X^{k,m}_- &:= -2iy \bigl(y \partial_{\ov\tau} + v \partial_{\ov z} \bigr)
\text{,}
\\[6pt]
   Y^{k,m}_+ &:= i \partial_z -4\pi m \frac{v}{y} 
\text{,}&
   Y^{k,m}_- &:= -iy \partial_{\ov z}
\text{,}
\\[6pt]
   X^{\sk;\, k,m}_+ &:= 2i \bigl(y^2 \partial_\tau + y v \partial_z + 2\pi i m v^2\bigr)
   + \tfrac{1}{2} y 
\text{,} & X^{\sk;\, k,m}_- &:= -2i \Big(\partial_{\ov\tau} + \frac{v}{y} \partial_{\ov z} \Big) + \bigl(k - \tfrac{1}{2}\bigr) \frac{1}{y}
\\[6pt]
   Y^{\sk;\, k,m}_+ &:= i y \partial_z -4\pi m v 
\text{,}&
   Y^{\sk;\, k,m}_- &:= -i \partial_{\ov z}
\text{.}
\end{align*}

\vspace{1ex}

The following proposition summarizes their properties.
\begin{proposition}
(\cite{BS, Raum-thesis}) If $A\in G^J$ and $\phi\in \C^{\infty}(\HH \times\C)$, then
\begin{align*}
   X^{k,m}_{\pm}\left(\phi \big|_{k,m}\, A\right) &= \left(X^{k,m}_{\pm}\phi\right) \big|_{k\pm 2,m}\, A 
\text{,}
&
   Y^{k,m}_{\pm}\left(\phi \big|_{k,m}\, A\right) &= \left(Y^{k,m}_{\pm}\phi\right) \big|_{k\pm 1,m}\, A 
\text{,}
\\[6pt]
   X^{\sk;\,k,m}_{\pm}\left(\phi \big|_{k,m}^{sk}\, A\right) &= \left(X^{\sk;\,k,m}_{\pm}\phi\right) \big|_{k\mp 2,m}^{sk}\, A 
\text{,}
&
   Y^{\sk;\,k,m}_{\pm}\left(\phi \big|_{k,m}^{sk}\, A\right)&= \left(Y^{\sk;\,k,m}_{\pm}\phi\right) \big|_{k\mp 1,m}^{sk}\, A
\text{.}
\end{align*}
\end{proposition}

\vspace{1ex}

The Casimir operator with respect to the action in (\ref{Jacobi-slash}) is given by
\begin{align}
\label{Casimir}
  \cC_{k,m}
:=
{} &
  2 X_+^{k-2,m} X_-^{k,m} - \tfrac{1}{2\pi m}\bigl( X_+^{k-2,m} Y_-^{k-1,m}Y_-^{k,m} - Y_+^{k-1,m}Y_+^{k-2,m} X_-^{k,m} \bigr)
\\[4pt]
\nonumber
{} &
  + \tfrac{1}{2\pi m}(k - 2)\, Y_+^{k-1,m}Y_-^{k,m}
\text{,}
\end{align}
and the Casimir operator with respect to the action in (\ref{skew-slash}) is given by (normalized as in \cite{BRR-Kohnen-limit}) 
$$
\cC_{k,m}^{\sk}:=8\pi i m\left(y^{\frac{1}{2}-k}  \cC_{1-k,m}y^{k-\frac{1}{2}}\right)+2k-1
$$
(see also \cite{BS, Ameya, B-R-Maass-Jacobi, Con-Raum, Raum-thesis}).  

Throughout, we adopt the following terminology.  A real-analytic $\phi:\HH\times\C\rightarrow\C$ is {\it semi-meromorphic} if $\phi(\tau, \,\cdot\,)$ is meromorphic with isolated singularities for all $\tau \in \HS$.  In this case $\phi$ is annihilated by $Y_-^{k,m}$ or by $Y_-^{\sk;k,m}$.   Moreover, we call a semi-meromorphic function that has no singularities \emph{semi-holomorphic}.  Finally, if $\phi$ is annihilated by the {\it Heisenberg Laplace operator}
\begin{gather}
\label{eq:heisenberg-laplace}
\lapH_m := Y_+^{k-1,m} Y_-^{k,m} = Y^{\sk;\, k+1,m}_+ Y^{\sk;\, k,m}_-,
\end{gather}
then $\phi$ is {\it Heisenberg~harmonic (\Hharmonic)}.  Note that the differential operator $\Delta_0$ in~\cite[p.~38]{BS} is very similar to $\lapH_m$.

\section{${\rm H}$-harmonic Maass-Jacobi forms}
\label{sec:harmonicjacobiforms}

The understanding of Maass-Jacobi forms is evolving with connections to different areas of mathematics and physics.  Maass-Jacobi forms were first introduced by Berndt and Schmidt \cite{BS}, and then more thoroughly investigated by Pitale \cite{Ameya}.  The first and third author \cite{B-R-Maass-Jacobi} extended Pitale's approach even further to include weak Maass-Jacobi forms.  The theory in \cite{B-R-Maass-Jacobi} includes new examples in the form of semi-holomorphic Poincar\'{e} series, but lacked new examples that are not holomorphic in $z$.  The notion of harmonic Maass-Jacobi forms in \cite{B-R-Maass-Jacobi} is also quite general, and refinements of the definition of Maass-Jacobi forms are needed.  In this section, we introduce the space $\JJHkm$ of Heisenberg~harmonic (\Hharmonic) Maass-Jacobi forms, and we give the differential operators that are needed to define its subspaces in Table~\ref{tab:maass_jacobi_spaces}.  These subspaces provide the desired refinements of Maass-Jacobi forms.

First we recall the weight $k$ hyperbolic Laplacian
\begin{gather*}
  \Delta_k
:=
  -4 y^2 \partial_{\tau} \partial_{\ov \tau} + 2 k i y\partial_{\ov \tau}
\text{,}
\end{gather*}
which is needed for the definition of harmonic weak Maass forms:
\begin{definition}[Harmonic weak Maass forms]
A harmonic weak Maass form of weight $k$ on a congruence subgroup $\Gamma \subset \Mp{2}(\ZZ)$ is a real-analytic function $f :\, \HS \rightarrow \CC$ satisfying the following conditions:
\begin{enumerate}[(1)]
\item For all $\big(\left(\begin{smallmatrix} a & b \\ c & d \end{smallmatrix}\right), \sqrt{c \tau + d} \, \big) \in \Gamma$, we have $f\big( \frac{a \tau + b}{c \tau + d} \big) = \sqrt{c \tau + d}^{2 k} \, f(\tau)$.\vspace{1ex}
\item We have that $\Delta_k(f) = 0$.\vspace{1ex}
\item The function $f$ has at most linear exponential growth at all cusps of $\Gamma$.
\end{enumerate} 
Let $\bbM_k$ denote the space of harmonic Maass forms of weight $k$, and denote its subspace of weakly holomorphic modular forms by ${\rm M}_{k}^! \subset \bbM_k$.
\end{definition}

The next definition allows us to define Jacobi forms with singularities in Definition~\ref{def:maassjacobiforms}.

\begin{definition}
\label{def:real_analytic_poles}
We say that a function $\phi :\, \RR^n \rightarrow \CC$ has a \emph{singularity of type $f g^{-1}$ at $x \in \RR^n$} if there are non-zero real-analytic functions $f$ and $g$ on a neighborhood $U \subset \RR^n$ of $x$ such that $\phi - f g^{-1}$ can be extended to a real-analytic function on $U$.  In addition, if $\phi$ is defined on a space with a complex structure and if $f$ and $g$ are holomorphic, then we say that $\phi$ has a singularity of meromorphic type.
\end{definition}
\begin{remark}
Functions whose natural domain of definition (see~\cite{Nishino} for details) are multi sheeted lead to singularities that are not as in Definition~\ref{def:real_analytic_poles}.  Prominent examples are logarithmic singularities and roots.
\end{remark}

We now improve the definition of harmonic Maass-Jacobi forms in \cite{B-R-Maass-Jacobi}.

\begin{definition}[${\rm H}$-harmonic Maass-Jacobi forms]
\label{def:maassjacobiforms}
Let $\phi :\, \HS \times \CC \rightarrow \CC$ be a real-analytic function except for possible singularities of type $f g^{-1}$, where $f$ and $g$ are real-analytic, such that the singularities of $\phi(\tau, \,\cdot\,)$ are isolated for every $\tau \in \HS$.  Then $\phi$ is an \Hharmonic\ Maass-Jacobi form of weight $k$ and index $m$ if the following conditions are satisfied:

\begin{enumerate}[(1)]
\item For all $A \in \JacF$, we have $\phi \big|_{k, m} A = \phi$.\vspace{1ex}
\item \label{it:harmonic_condition_Casimir} We have that $\cC_{k,m}(\phi)=0$.\vspace{1ex}
\item \label{it:hharmonic_condition_Laplace} We have that $\lapH_m(\phi)=0$.\vspace{1ex}
\item \label{it:growth_condition} For every $\alpha, \beta \in \QQ$ such that $\{(\tau, \alpha\tau + \beta) \,:\, \tau \in \HS\}$ is not a polar divisor of $\phi$, we have that $\phi (\tau, \alpha\tau+\beta) = O\big( e^{a y} \big)$ as $y \rightarrow \infty$ for some $a>0$.
\end{enumerate}
We write $\JJHkm$ for the space of such forms.
\end{definition}

\begin{remark}
We call condition~(\ref{it:growth_condition}) in the previous definition the growth condition.  A priori, it is not clear if there are $\alpha, \beta$ such that the function $\phi(\tau, \alpha\tau + \beta)$ has singularities for arbitrary large~$y$.  However, Proposition~\ref{prop:singularities_of_maass_jacobi_forms} shows that this is not the case.  Note that we need the growth condition~(\ref{it:growth_condition}) only in the proof of Theorems~\ref{thm:intro_thetadecomposition} and~\ref{thm:hharmonic_maass_jacobi_forms_thetadecomposition}, in order to establish a relation to harmonic weak Maass forms, which also satisfy a certain growth condition.
\end{remark}

In the following we define analogs of Bruinier and Funke's \cite{B-F-Duke04} operator $\xi_k$, which are needed to characterize the spaces of Jacobi forms in Table~\ref{tab:maass_jacobi_spaces}.  Set 
\begin{align}
\label{eq:xiJH}
  \xiJHkm (\phi)
&:=
  \sqrt{-my}^{\,-1} \exp\big( -4 \pi m \tfrac{v^2}{y} \big)\; \ov{Y_-^{k,m} (\phi)}
\quad\text{and}
\\[6pt]
\label{eq:xiJHsk}
  \xiJHskkm(\phi)
&:=
  \sqrt{-my} \exp\big( -4 \pi m \tfrac{v^2}{y} \big)\; \ov{Y_-^{\sk;k,m} (\phi)}
\text{.}
\end{align}
The operators $\xiJHkm$ and $\xiJHskkm$ are covariant with respect to the actions in (\ref{Jacobi-slash}) and (\ref{skew-slash}): If $\phi$ is a smooth function on $\HH\times\C$ and $A\in G^J$, then
\begin{align}
\label{eq:xi_equiv}
 \big(\xiJHkm (\phi) \big)\big|^{\sk}_{k,-m}A
&=\xiJHkm \Big( \phi\big|_{k,m}A \Big)
  \quad\text{and}
\\[6pt]
 \big(\xiJHskkm (\phi) \big)\big|_{k,-m}A
&=\xiJHskkm \Big( \phi\big|^{\sk}_{k,m}A \Big)\text{.}
\end{align}
Recall that the weight $k$ hyperbolic Laplacian factors as $\Delta_k=-\xi_{2-k} \circ \xi_k$.  Similarly, one finds that
\begin{gather}
\label{eq:lapH_factor}  
  \lapH_m
=
  \xiJHsk_{k,-m} \circ \xiJHkm=\xiJH_{k,-m}\circ \xiJHsk_{k,m}
\text{.}
\end{gather}

From \cite{B-R-Maass-Jacobi} and  \cite{BRR-Kohnen-limit} recall the definitions
\begin{align}
\label{eq:xiJ}
  \xiJkm (\phi)
&:=
  y^{k-5/2} \Big( X_-^{k,m} (\phi) - \tfrac{1}{4\pi m} Y_-^{k-1,m}Y_-^{k,m} (\phi) \Big)
\quad\text{and}
\\[6pt]
\label{eq:xiJsk}
  \xiJskkm \big(\phi\big)
&:=
  y^{k-5/2} \Big( X_+^{\sk;k,m} (\phi) - \tfrac{1}{4\pi m} Y_+^{\sk;k+1,m} Y_+^{\sk;k,m} (\phi) \Big)
=
  \tfrac{1}{4\pi m} y^{k-\frac{1}{2}} \,L_m (\phi)
\text{,}
\end{align}
where $L_m:=8 \pi im \partial_{\tau}-\partial_{z}^2$ is the usual heat-operator.  
The operators $\xiJkm$ and $\xiJskkm$ are also covariant with respect to the actions in (\ref{Jacobi-slash}) and (\ref{skew-slash}): If $\phi$ is a smooth function on $\HH\times\C$ and $A\in G^J$, then
\begin{align}
\label{eq:xiJ_equiv}
  \big(\xiJkm (\phi) \big)\big|^{\sk}_{3-k,m}A
&=
  \xiJkm \Big(\phi\big|_{k,m}A \Big)
  \quad\text{and}
\\[6pt]
  \big(\xiJskkm (\phi) \big)\big|_{3-k,m}A
&=
  \xiJskkm \Big(\phi\big|^{\sk}_{k,m}A\Big)
\text{.}
\end{align}
The actions of the Casimir operators simplify when applied to semi-meromorphic functions.  Precisely, if $\phi$ is semi-meromorphic, then one verifies that
\begin{gather}
\label{eq:semimeromorphic_casimir}
  \cC_{k,m} (\phi)
=
   2 \, \xiJsk_{3 - k, m} \circ \xiJkm (\phi)
\quad\text{and}\quad
  \cC^\sk_{k,m} (\phi)
=
   2 \, \xiJ_{3 - k, m} \circ \xiJskkm (\phi)
\text{.}
\end{gather}

We also consider the space $\JJHskkm$ of \Hharmonic\ skew-Maass-Jacobi forms of weight $k$ and index $m$.  This space consists of functions $\phi$ as in Definition \ref{def:maassjacobiforms}, where conditions $(1)$ and $(2)$ are replaced by
\begin{enumerate}[(1')]
\item For all $A \in \JacF$, we have $\phi \big|_{k,m}^{sk} A = \phi$.\vspace{1ex}
\item \label{it:harmonic_condition_Casimir_skew} We have that $\cC^\sk_{k,m}(\phi)=0$.\vspace{1ex}
\end{enumerate}
The operators $\xiJHskkm$ and $\xiJskkm$ allow us to define skew-Maass-Jacobi versions of the spaces in Table~\ref{tab:maass_jacobi_spaces}.  Specifically, the forms in $\JJHskkm$ that are annihilated by $\xiJskkm$ and $\xiJHskkm$ are denoted by $\JHskkm$ and $\JJskkm$, respectively, and $\Jskkm:=\JHskkm\cap\JJskkm$.  In this paper, we will encounter only the following two subspaces of \Hharmonic\ skew-Maass-Jacobi forms without singularities:  The space $\holJskkm \subseteq \Jskkm$, which contains Skoruppa's skew-holomorphic Jacobi forms (see~\cite{Sko-mpim, Sko-Invent90}), and $\holJJskkm \subseteq \JJskkm$.  Note that Corollary~\ref{cor:semimeromorphic_skew_maass_jacobi_forms} will show that $\JJskkm = \holJJskkm$, and that $\holJJskkm = \{0\}$ if $m < 0$.

Our next task is to describe the Fourier series expansions of \Hharmonic\ Maass-Jacobi forms.  For this purpose we will need the lower incomplete Gamma-function $\gamma(s,x):=\int_0^xt^{s-1}e^{-t}\ dt$ and the function 
\begin{gather}
\label{eq:H}
H(w):=e^{-w}\int_{-2w}^{\infty}t^{\frac{1}{2}-k}e^{-t}\ dt
\text{.}
\end{gather} 
Observe that $H(w)$ converges for $k<\frac{3}{2}$ and has a holomorphic continuation in $k$ if $w \ne 0$.  If $w<0$, then $H(w)=e^{-w}\, \Gamma(\frac{3}{2}-k,-2w)$ (see also page 55 of \cite{B-F-Duke04}), where  $\Gamma(s,x):=\int_x^{\infty}t^{s-1}e^{-t}\ dt$ is the upper incomplete Gamma-function. Throughout, we write $q:=e^{2 \pi i \tau}$ and $\zeta:=e^{2\pi iz}$. 

\begin{proposition}
\label{prop:hharmonic_fourierexpansion}
Suppose $\phi\in\JJHkm$ has a local Fourier series expansion of the form
\begin{gather}
\label{eq:H-Harmonic_Fourier_coefficients}
\sum_{\substack{n, r \in\Z \\D=4mn-r^2  } } c(n,r; \,y,v)q^n\zeta^r
\text{.}
\end{gather}
If $m>0$, then $c(n,r; \,y,v)$ lies in the $2$-dimensional vector space spanned by $c_1(n,r;\, y, v)$ and $c_2(n,r;\, y, v)$ below.  If $m < 0$, then $c(n,r; \,y,v)$ lies in the $4$-dimensional vector space spanned by $c_1(n,r;\, y, v), \ldots, c_4(n,r;\, y, v)$ below. If $D \ne 0$, then
\begin{align*}
  c_1(n, r;\, y, v)
&=
  1
\text{,}
\qquad\quad
  c_2(n, r;\, y, v)
=
  H\Big( \frac{\pi D y}{2 m } \Big) \exp\Big( \frac{\pi D y}{2 m} \Big)
\text{,}
\\[6pt]
  c_3(n, r;\, y, v)
&=
  \sgn\big( r + 2 m \tfrac{v}{y} \big) \,
  \gamma\Big( \tfrac{1}{2}, \tfrac{- \pi y}{m} \big( r + 2 m \tfrac{v}{y} \big)^2 \Big)
\text{,}
\\[6pt]
  c_4(n, r;\, y, v)
&=
  H\Big( \frac{\pi D y}{2 m } \Big) \exp\Big( \frac{\pi D y}{2 m} \Big) \;
  \sgn\big( r + 2 m \tfrac{v}{y} \big) \,
  \gamma\Big( \tfrac{1}{2}, \tfrac{- \pi y}{m} \big( r + 2 m \tfrac{v}{y} \big)^2 \Big) 
\text{.}
\end{align*}
If $D = 0$, then
\begin{align*}
  c_1(n, r;\, y, v)
&=
  1
\text{,}
\qquad\quad
  c_2(n, r;\, y, v)
=
   y^{\frac{3}{2} - k}
\text{,}
\\[6pt]
  c_3(n, r;\, y, v)
&=
  \sgn\big( r + 2 m \tfrac{v}{y} \big) \,
  \gamma\Big( \tfrac{1}{2}, \tfrac{- \pi y}{m} \big( r + 2 m \tfrac{v}{y} \big)^2 \Big)
\text{,}
\\[6pt]
  c_4(n, r;\, y, v)
&=
  y^{\frac{3}{2} - k} \;
  \sgn\big( r + 2 m \tfrac{v}{y} \big) \,
  \gamma\Big( \tfrac{1}{2}, \tfrac{- \pi y}{m} \big( r + 2 m \tfrac{v}{y} \big)^2 \Big)
\text{.}
\end{align*}
\end{proposition}

\begin{proof}
It is easy to verify that all $c_i(n,r;\, y, v)q^n\zeta^r$ are in the kernels of $\cC_{k,m}$ and $\lapH_m$. Moreover, for each $(n,r)$ the differential equation $\cC_{k,m}\bigl(c(n,r; \,y,v)q^n\zeta^r\bigr)=0$ has at most four linear independent solutions, which can be seen as follows:  For fixed $n, r$, and $y$ the differential equation for $c(n,r; \,y,v)$ arising from $\lapH_m\big( c(n,r; \,y,v) \, q^n\zeta^r \big) = 0$ has order~$2$, hence leading to at most two linear independent solutions $f_i(n, r;\, y,v)$ with coefficients $d_i(n, r;\, y)$ ($i = 1, 2$).  Imposing $\cC_{k,m}\big( d_i(n, r;\, y) f_i(n, r;\, y,v) \, q^n \zeta^r \big) = 0$ yields a differential equation of order~$2$ for each $d_i(n, r;\, y)$.  Thus, there are at most two linear independent solutions for each $d_i(n, r;\, y)$, and hence at most four linear independent solutions for $c(n,r; \,y,v)$.

In Section \ref{sec:singluarjacobiforms} we will prove Corollary~\ref{cor:singularities_of_maass_jacobi_forms}, which implies that $\psi:=\xiJHkm(\phi)$ has no singularities.  In particular, $\psi$ is a semi-holomorphic skew-Maass-Jacobi form of index $-m$ (observe (\ref{eq:xi_equiv}) and (\ref{eq:lapH_factor})), and if $\psi\not=0$, then $-m > 0$.  Thus, if $m>0$, then $c(n,r; \,y,v)$ in (\ref{eq:H-Harmonic_Fourier_coefficients}) is a linear combination of the semi-holomorphic solutions $c_1(n,r;\, y, v)$ and $c_2(n,r;\, y, v)$.
\end{proof}

The situation for \Hharmonic\ skew-Maass-Jacobi forms is very similar.  We only record the result on their Fourier coefficients and omit the proof.

\begin{proposition}
\label{prop:skew-hharmonic_fourierexpansion}
Let $\phi\in\JJHskkm$ such that $\xiJHskkm(\phi)$ has no singularities and suppose that $\phi$ has a local Fourier series expansion of the form
\begin{gather}
\sum_{\substack{n, r \in\Z \\D=4mn-r^2  } } c^\sk(n,r; \,y,v)q^n\zeta^r
\text{.}
\end{gather}
If $m>0$, then $c^\sk(n,r; \,y,v)$ lies in the $2$-dimensional vector space spanned by $c^\sk_1(n,r;\, y, v)$ and $c^\sk_2(n,r;\, y, v)$ below.  If $m < 0$, then $c^\sk(n,r; \,y,v)$ lies in the $4$-dimensional vector space spanned by $c^\sk_1(n,r;\, y, v), \ldots, c^\sk_4(n,r;\, y, v)$ below. If $D \ne 0$, then
\begin{align*}
  c_1^\sk (n,r; \,y,v)
&=
  \exp\Big( \frac{\pi D y}{m} \Big)
\text{,}
\qquad\quad
  c_2^\sk (n,r; \,y,v)
=
  H\Big( \frac{- \pi D y}{2 m } \Big) 
  \exp\Big( \frac{\pi D y}{2 m} \Big)
\text{,}
\\[6pt]
  c_3^\sk (n,r; \,y,v)
&=
  \exp\Big( \frac{\pi D y}{m} \Big) \;
  \sgn\big( r + 2 m \tfrac{v}{y} \big) \,
  \gamma\Big( \tfrac{1}{2}, \tfrac{- \pi y}{m} \big( r + 2 m \tfrac{v}{y} \big)^2 \Big)
\text{,}
\\[6pt]
  c_4^\sk (n,r; \,y,v)
&=
  H\Big( \frac{- \pi D y}{2 m } \Big)
  \exp\Big( \frac{\pi D y}{2 m} \Big) \;
  \sgn\big( r + 2 m \tfrac{v}{y} \big) \,
  \gamma\Big( \tfrac{1}{2}, \tfrac{- \pi y}{m} \big( r + 2 m \tfrac{v}{y} \big)^2 \Big)
\text{.}
\end{align*}
If $D = 0$, then
\begin{align*}
  c_1^\sk (n,r; \,y,v)
&=
  1
\text{,}
\qquad\quad
  c_2^\sk (n,r; \,y,v)
=
  y^{\frac{3}{2} - k}
\text{,}
\\[6pt]
  c_3^\sk (n,r; \,y,v)
&=
  \sgn\big( r + 2 m \tfrac{v}{y} \big) \,
  \gamma\Big( \tfrac{1}{2}, \tfrac{- \pi y}{m} \big( r + 2 m \tfrac{v}{y} \big)^2 \Big)
\text{,}
\\[6pt]
  c_4^\sk (n,r; \,y,v)
&=
  y^{\frac{3}{2} - k} \,
  \sgn\big( r + 2 m \tfrac{v}{y} \big) \,
  \gamma\Big( \tfrac{1}{2}, \tfrac{- \pi y}{m} \big( r + 2 m \tfrac{v}{y} \big)^2 \Big)
\text{.}
\end{align*}
\end{proposition}

The $\xi$-operators in (\ref{eq:xiJH}), (\ref{eq:xiJHsk}), (\ref{eq:xiJ}), and (\ref{eq:xiJsk}) provide the following interplay between the Fourier coefficients in Proposition~\ref{prop:hharmonic_fourierexpansion} and Proposition~\ref{prop:skew-hharmonic_fourierexpansion}.

\begin{proposition}
\label{prop:images_of_xi}
Let $c_i(n, r;\, y, v)$ and $c^\sk_i(n, r;\, y, v)$ be the Fourier coefficients in Proposition~\ref{prop:hharmonic_fourierexpansion} and Proposition~\ref{prop:skew-hharmonic_fourierexpansion}, respectively.  With an abuse of notation we write $\widetilde{c}_i:=\widetilde{c}_i[k, m, n, r]:= c_i(n,r; \,y,v)q^n\zeta^r$ and $\widetilde{c}_i^\sk:=\widetilde{c}_i^\sk[k,m,n,r]:= c_i^\sk(n,r; \,y,v)q^n\zeta^r$.  If $D \not= 0$, then 
\begin{align*}
 \xiJkm \bigl(\widetilde{c}_1\bigr) &= 0
\text{,}
&
  \xiJkm \bigl(\widetilde{c}_2\bigr) &= -\left(\tfrac{-\pi D}{m}\right)^{\frac{3}{2}-k} \; \widetilde{c}_1^\sk[3-k, m,n,r]
\text{,}
\\[3pt] 
  \xiJkm \bigl(\widetilde{c}_3\bigr) &= 0
\text{,}
&
  \xiJkm \bigl(\widetilde{c}_4\bigr) &= -\left(-\tfrac{\pi D}{m}\right)^{\frac{3}{2}-k} \; \widetilde{c}_3^\sk[3-k,m,n,r]
\text{,}
\\[15pt]
  \xiJHkm \bigl(\widetilde{c}_1\bigr) &= 0
\text{,}
&
  \xiJHkm \bigl(\widetilde{c}_3\bigr) &= -2\sqrt{\pi} \; \widetilde{c}_1^\sk[k, -m,-n,-r]
\text{,}
\\[3pt]
  \xiJHkm \bigl(\widetilde{c}_2\bigr) &= 0
\text{,}
&
  \xiJHkm \bigl(\widetilde{c}_4\bigr) &= -2\sqrt{\pi} \; \widetilde{c}_2^\sk[k, -m,-n,-r]
\text{,}
\\[15pt]
  \xiJskkm\bigl (\widetilde{c}_1^\sk\bigr) &= 0
\text{,}
&
  \xiJskkm (\widetilde{c}_2^\sk) &= -\left(\tfrac{\pi D}{m}\right)^{\frac{3}{2}-k} \; \widetilde{c}_1[3-k, m,n,r]
\text{,}
\\[3pt]
  \xiJskkm\bigl (\widetilde{c}_3^\sk\bigr) &= 0
\text{,}
&
  \xiJskkm (\widetilde{c}_4^\sk) &= -\left(\tfrac{\pi D}{m}\right)^{\frac{3}{2}-k}  \; \widetilde{c}_3[3-k, m,n,r]
\text{,}
\\[15pt]
 \xiJHskkm \bigl(\widetilde{c}_1^\sk\bigr) &= 0
\text{,}
&
  \xiJHskkm \bigl(\widetilde{c}_3^\sk\bigr) &= -2\sqrt{\pi} \; \widetilde{c}_1[k,-m, -n,-r] 
\text{,}
\\[3pt]
 \xiJHskkm \bigl(\widetilde{c}_2^\sk\bigr) &= 0
\text{,}
&
 \xiJHskkm \bigl(\widetilde{c}_4^\sk\bigr) &= -2\sqrt{\pi} \; \widetilde{c}_2[k, -m,-n,-r]
 \text{.}
\end{align*}

\vspace{1ex}

If $D= 0$, then 
\begin{align*}
 \xiJkm \bigl(\widetilde{c}_1\bigr) &= 0
\text{,}
&
  \xiJkm \bigl(\widetilde{c}_2\bigr) &= (\tfrac{3}{2}-k) \; \widetilde{c}_1^\sk[3-k, m,n,r]
\text{,}
\\[3pt] 
  \xiJkm \bigl(\widetilde{c}_3\bigr) &= 0
\text{,}
&
  \xiJkm \bigl(\widetilde{c}_4\bigr) &= (\tfrac{3}{2}-k) \; \widetilde{c}_3^\sk[3-k,m,n,r]
\text{,}
\\[15pt]
  \xiJHkm \bigl(\widetilde{c}_1\bigr) &= 0
\text{,}
&
  \xiJHkm \bigl(\widetilde{c}_3\bigr) &= -2\sqrt{\pi} \; \widetilde{c}_1^\sk[k, -m, -n,-r]
\text{,}
\\[3pt]
  \xiJHkm \bigl(\widetilde{c}_2\bigr) &= 0
\text{,}
&
  \xiJHkm \bigl(\widetilde{c}_4\bigr) &= -2\sqrt{\pi} \; \widetilde{c}_2^\sk[k, -m,-n,-r]
\text{,}
\\[15pt]
  \xiJskkm\bigl (\widetilde{c}_1^\sk\bigr) &= 0
\text{,}
&
  \xiJskkm (\widetilde{c}_2^\sk) &= (\tfrac{3}{2}-k) \; \widetilde{c}_1[3-k, m,n,r]
\text{,}
\\[3pt]
  \xiJskkm\bigl (\widetilde{c}_3^\sk\bigr) &= 0
\text{,}
&
  \xiJskkm (\widetilde{c}_4^\sk) &= (\tfrac{3}{2}-k) \; \widetilde{c}_3[3-k, m,n,r]
\text{,}
\\[15pt]
 \xiJHskkm \bigl(\widetilde{c}_1^\sk\bigr) &= 0
\text{,}
&
  \xiJHskkm \bigl(\widetilde{c}_3^\sk\bigr) &= -2\sqrt{\pi} \; \widetilde{c}_1[k,-m, -n,-r]
\text{,}
\\[3pt]
 \xiJHskkm \bigl(\widetilde{c}_2^\sk\bigr) &= 0
\text{,}
&
 \xiJHskkm \bigl(\widetilde{c}_4^\sk\bigr) &=-2\sqrt{\pi} \; \widetilde{c}_2[k, -m,-n,-r]
 \text{.}
\end{align*}

\end{proposition}
\begin{proof}
Observe the covariance properties of the $\xi$-operators in (\ref{eq:xi_equiv}) and (\ref{eq:xiJ_equiv}). All identities of the proposition follow from straightforward computations.
\end{proof}

Let $\cF\cE^{\rm J}$ denote the space of real-analytic functions $\HS \times \CC \rightarrow \CC$ that admit a local Fourier series expansion at some point.   \Hharmonic\ Maass-Jacobi forms and \Hharmonic\ skew-Maass-Jacobi forms that have local Fourier series expansions are connected in a natural way via the $\xi$-operators in (\ref{eq:xiJH}), (\ref{eq:xiJHsk}), (\ref{eq:xiJ}), and (\ref{eq:xiJsk}), and the following corollary is a direct consequence of Proposition~\ref{prop:hharmonic_fourierexpansion}, Proposition~\ref{prop:skew-hharmonic_fourierexpansion}, and Proposition~\ref{prop:images_of_xi}.

\begin{corollary}
\label{cor:commuting_xi_operators}
The following diagrams are commutative:
\begin{gather*}
\xymatrix{
  \cF\cE^{\rm J} \cap \JHsk_{3 - k, m}  \ar[d]^{\xiJHsk_{3 - k, m}} & \cF\cE^{\rm J} \cap \JJHkm        \ar[l]_{\xiJkm} \ar[d]^{\xiJHkm} \\
  \cF\cE^{\rm J} \cap \J_{3 - k, -m}                            & \cF\cE^{\rm J} \cap \JJsk_{k, -m}  \ar[l]_{\xiJsk_{k, -m}}
}
\qquad
\xymatrix{
  \cF\cE^{\rm J} \cap \JH_{3 - k, m}  \ar[d]^{\xiJH_{3 - k, m}} & \cF\cE^{\rm J} \cap \JJHskkm   \ar[l]_{\xiJskkm} \ar[d]^{\xiJHskkm} \\
  \cF\cE^{\rm J} \cap \Jsk_{3 - k, -m}                      & \cF\cE^{\rm J} \cap \JJ_{k, -m}  \ar[l]_{\xiJ_{k, -m}}
}
\end{gather*}
\end{corollary}

\begin{remark}
\ 
\begin{enumerate}[(1)]
\item Any \Hharmonic\ Maass-Jacobi form that has non-moving singularities admits a local Fourier series expansion.
\item It will follow from Theorem~\ref{thm:structure_maass_jacobi_forms} and Proposition~\ref{prop:singular_skew_maass_jacobi_forms} that the left diagram in~Corollary \ref{cor:commuting_xi_operators} is already commutative when omitting the intersections with~$\cF\cE^{\rm J}$.
\end{enumerate}
\end{remark}

\section{Maass-Jacobi forms with singularities}
\label{sec:singluarjacobiforms}

In this section, we investigate the singularities of \Hharmonic\ (skew)-Maass-Jacobi forms and prove Theorem~\ref{thm:structure_maass_jacobi_forms}~(\ref{thm:it:semimeromorphic_maass_jacobi_forms}). A key ingredient is the next proposition, which relies on the theory of several complex variables.
\begin{proposition}
\label{prop:laurent_expansion}
Let $\phi :\, \HS \times \CC \rightarrow \CC$ be a real-analytic function except for possible singularities of type $f g^{-1}$, where $f$ and $g$ are real-analytic, such that the singularities of $\phi(\tau, \,\cdot\,)$ are isolated for every $\tau \in \HS$. Suppose that $\lapH_m(\phi)=0$ for some half-integer $m$.  Then either $\phi$ has no singularities, or there exist $\tau_0 \in \HS$ and real-analytic $z_0 :\, \HS \rightarrow \CC$ such that $\phi$ has a Laurent series expansion of the form
\begin{gather}
\label{eq:laurent_expansion_general}
  \sum_{n > -N, n' > -N'} c_{n,n'}(\tau) \big(z - z_0(\tau)\big)^n \ov{\big(z - z_0(\tau)\big)}^{\, n'}
\end{gather}
around $\big( \tau_0, z_0(\tau_0) \big) \in \HS \times \CC$.  In particular, if $\phi$ is semi-meromorphic, then its Laurent series expansion around $\big( \tau_0, z_0(\tau_0) \big) \in \HS \times \CC$ equals
\begin{gather}
\label{eq:laurent_expansion_semimeromorph}
  \sum_{n > -N} c_n(\tau) \big(z - z_0(\tau)\big)^n
\text{.}
\end{gather}
\end{proposition}
\begin{proof}
Suppose that $\phi$ has a singularity at $\big(\tau_0, z_0(\tau_0)\big)$, where $\tau_0 \in \HS$ and  $z_0(\tau_0) \in \CC$.  It suffices to show that there are open sets $\tau_0 \in U \subset \HS$ and $z_0(\tau_0) \in V \subset \CC$, and a real-analytic function $z_0 :\, U \rightarrow \CC$ such that for $\tau \in U$ the singularities of $\phi(\tau, \,\cdot\,)$ in a neighborhood of $z_0(\tau_0)$ lie exactly at $z_0(\tau)$ and have the same multiplicities for all $\tau \in U$.

We first assume that $\phi$ is semi-meromorphic.  Choose a neighborhood $U \times V$ of $\big(\tau_0, z_0(\tau_0)\big)$, small enough such that $\phi$ can be considered as a meromorphic function of three complex variables $x, y \in \CC_j := \RR + j \RR$ ($j^2 = -1$) and $z \in \CC$ restricted to $(\tau, z) = \big(x + i y, z \big) \in U \times V$ with $x, y \in \RR$.  We can write $\phi |_{U \times V}$ as a quotient of two holomorphic functions $f(x, y, z)$ and $g(x, y, z)$ in three variables, after possibly shrinking $U$ and $V$ (see the treatment of the Poincar\'e problem in~\cite[Proposition 3.1, Theorem 3.9]{Nishino}).  We may also assume that $f$ and $g$ are coprime, i.e., there is no open set~$W$ such that the sets of zeros of $f|_W$ and $g|_W$ are equal.

Since $\phi(\tau, \,\cdot\,)$ has isolated singularities, we can apply the Weierstrass Preparation Theorem (see~\cite[Theorem~2.1]{Nishino}) to $g$.  We find that the singularities of $\tau \mapsto \phi(\tau, z)$ are given by a product of powers of pairwise distinct irreducible pseudo polynomials (i.e., polynomials in $z$ with coefficients that are holomorphic functions of $x$ and $y$) $p_1(x,y;z),\ldots,p_l(x,y;z)$ for some $l$ after possibly shrinking $U$ and $V$ further.  Since these polynomials are coprime, one can move $\tau_0$ within $U$ (which may be needed if $l > 1$) and then shrink $U$ and $V$ even further such that finally $p_1(x,y;z)^r \phi(\tau, \,\cdot\,)$ has a holomorphic continuation on $U \times V$ for some $0 < r \in \ZZ$.  This proves the case when $\phi$ is semi-meromorphic.


If $\phi$ is not semi-meromorphic, then we will show that the locus of singularities of $\phi$ locally coincides with that of a semi-meromorphic function.  Write $\widetilde \phi$ for the image of $\phi$ under $\xiJHkm$ or $\xiJHskkm$.  Equation (\ref{eq:lapH_factor}) and the assumption that $\lapH_m(\phi)=0$ imply that $\widetilde \phi$ is semi-meromorphic, and $\widetilde \phi$ has singularities that can be described as above.  In particular, $\widetilde \phi$ has a local Laurent series expansion of the form
\begin{gather*}
  \sum_{n > -N} {\td c}_n(\tau) \big(z - z_0(\tau)\big)^n
\text{.}
\end{gather*}
For brevity we restrict to the case $\widetilde{\phi} = \xiJHkm (\phi)$; the case $\widetilde{\phi} = \xiJHskkm (\phi)$ is analogous.  Then $\partial_{\ov z} \, \phi$ has a local Laurent series expansion of the form
\begin{gather}
\label{eq:laurent_expansion_proof}
i\,\ov{\sqrt{\frac{-m}{y}}}\,\exp\Big(-\pi m \frac{(z - {\ov z})^2}{y}\Big) \sum_{n > -N}  \ov{{\td c}_n(\tau)} \big(\ov{z} - \ov{z_0(\tau)}\big)^n
\text{.}
\end{gather}
If $\tau\in\HS$ is fixed, then the assumptions on the singularities of $\phi$ guarantee that $\phi$ has a local Laurent series expansion in $z$ and ${\ov z}$. In particular, the coefficient of $\big(\ov{z} - \ov{z_0(\tau)}\big)^{-1}$ of the local Laurent series expansion of $\partial_{\ov z} \, \phi$ is zero, and one may formally integrate \eqref{eq:laurent_expansion_proof} with respect to $\ov{z}$.  This yields a real-analytic function $\phi^{\rm ra}$, which has a locally convergent Laurent series expansion as in \eqref{eq:laurent_expansion_general}, and which locally has the same locus of singularities as $\widetilde{\phi}$.  Moreover, $\phi - \phi^{\rm ra}$ is semi-meromorphic and by the above it has a local Laurent series expansion as in \eqref{eq:laurent_expansion_semimeromorph}. Thus $\phi$ has a local Laurent series expansion as in \eqref{eq:laurent_expansion_general}, which concludes the proof.
\end{proof}

Another crucial ingredient is the following proposition:
\begin{proposition}
\label{prop:singular_skew_maass_jacobi_forms}
There is no $\phi \in \JJHskkm$ that has a local Laurent series expansion with non-zero semi-meromorphic principal part
\begin{gather}
\label{eq:prop:singular_skew_maass_jacobi_forms}
\sum_{n = -N}^{-1} c_n(\tau) \big(z - z_0(\tau)\big)^n
\text{,}
\end{gather}
where $N > 0$ and $c_{-N} \ne 0$.
\end{proposition}
\begin{proof}
Let $\phi \in \JJHskkm$ with singularities, and assume that $\phi$ has a local Laurent series expansion as in Proposition~\ref{prop:laurent_expansion} with non-zero semi-meromorphic principal part
\begin{gather*}
  \cP(\tau, z)
:=
  \sum_{n = -N}^{-1} c_n(\tau) \big(z - z_0(\tau)\big)^n
\text{,}
\end{gather*}
where $N > 0$ and $c_{-N} \ne 0$.  By assumption, $\cP$ is semi-meromorphic and $\cC^\sk_{k,m}(\cP) = 0$.  The factorization~\eqref{eq:semimeromorphic_casimir} of $\cC^\sk_{k,m}$ for semi-meromorphic forms implies that $\xi^{\sk}_{k, m}(\cP)$ is meromorphic. In particular, $\partial_{\ov \tau}\, \xiJsk_{k,m}(\cP)=0$. Explicitly, we have
\begin{gather}
\label{eq:skew_singularities_apply_xiJsk}
  \xi^\sk_{k,m}(\cP)
=
  2i\,y^{k - \frac{1}{2}}
  \sum_{n = -N}^{-1} \bigg(
              \partial_\tau \Big( c_n(\tau) \big(z - z_0(\tau)\big)^n \Big)
            - \frac{n (n - 1)}{8 \pi im} c_n(\tau) \big(z - z_0(\tau)\big)^{n - 2} \bigg)
\text{.}
\end{gather}
%
%
%
We inspect the coefficients of $\big(z - z_0(\tau)\big)^{-N - 3}$ and $\big(z - z_0(\tau)\big)^{-N - 2}$ in the Laurent series expansion of $\partial_{\ov \tau} \, \xiJsk_{k,m}(\cP)$) and see that $z_0$ and $y^{k - \frac{1}{2}} c_{-N} (\tau)$ are holomorphic.
If $n < 0$, then an induction argument shows that 
\begin{gather*}
c_n(\tau) = \sum_{l \in \ZZ + \frac{1}{2}} y^l c_{n, l}(\tau),
\end{gather*}
where $c_{n, l}(\tau)$ is holomorphic and the sum is finite.  If $n = -N$, then this is true by the above.  Assume that the claim is true for all $n < n_0 < 0$.  Apply $\partial_{\ov \tau}$ to~\eqref{eq:skew_singularities_apply_xiJsk} to obtain
\begin{gather*}
  \partial_{\ov \tau}
 \Bigg(  2 i
  y^{k - \frac{1}{2}}
  \Big(  \big(\partial_{\tau} c_{n_0 - 2}\big) (\tau)
      - (n_0 - 2) \big(\partial_{\tau} z_0\big) (\tau) c_{n_0 - 1}(\tau)
       - \frac{n_0(n_0 - 1)}{8 \pi i m} c_{n_0}(\tau)
  \Big)
  \Bigg)
=
  0
\text{,}
\end{gather*}
which proves that $c_{n_0}(\tau)$ is of the required form.

Let $n_0 < 0$ be maximal such that $c_{n_0} \ne 0$.  Expand the coefficient of \mbox{$\big(z - z_0(\tau)\big)^{n_0}$} in the Laurent series expansion of $\partial_{\ov \tau}\xi^\sk_{k,m}(\cP)$ to find that $y^{k - \frac{1}{2}} \sum_{l \in \ZZ + \frac{1}{2}} \big( \tfrac{-il}{2} y^{l - 1} c_{n_0, l}(\tau) + y^l \partial_\tau c_{n_0, l}(\tau)\big)$ is holomorphic.  It is easy to see that this is only possible if $c_{n_0}(\tau) = c \, y^{\frac{3}{2}-k}$ for some $0\not=c\in\C$.

Note that $n_0 \ne -N$, since $y^{k - \frac{1}{2}} c_{n_0} (\tau)$ is not holomorphic.   Consider the coefficient of \mbox{$\big(z - z_0(\tau)\big)^{n_0-1}$} in the Laurent series expansion of $\partial_{\ov \tau}\xi^\sk_{k,m}(\cP)$ to discover that $-n_0 cy\partial_\tau z_0(\tau) + y^{k - \frac{1}{2}}\sum_{l \in \ZZ + \frac{1}{2}} \big( \tfrac{-il}{2} y^{l - 1} c_{n_0-1, l}(\tau) + y^l \partial_\tau c_{n_0-1, l}(\tau)\big)$ is holomorphic, which is only possible if $\partial_\tau z_0(\tau)$ is a polynomial, since the sum over $l$ is finite.


Let $A_S:=\left[\big(\left(\begin{smallmatrix}0 & -1\\1 & 0\end{smallmatrix}\right), \sqrt{\tau} \,\big), (0, 0)\right]\in\Gamma^{\rm J}$.  If $\big(\tau_0,z_0(\tau_0)\big)$ is a singularity of $\phi=\phi\,\big|_{k,m}^{sk} A_S$, then so is $\big(\widetilde{\tau_0},\widetilde{z_0}(\widetilde{\tau_0})\big)$, where $\big(\widetilde{\tau_0},\widetilde{z_0}(\tau)\big):=\big(\frac{-1}{\tau_0},\tau z_0\big(\frac{-1}{\tau}\big)\big)$. Moreover, $\phi$ has a local Laurent series expansion with non-zero semi-meromorphic principal part
\begin{gather*}
  \widetilde{\cP}(\tau, z)
:=
  \sum_{n = -N}^{-1} \widetilde{c}_n(\tau) \big(z - \widetilde{z}_0(\tau)\big)^n
\text{,}
\end{gather*}
where $\widetilde{c}_n = 0$ for $n > n_0$ and $\widetilde{c}_{n_0} \ne 0$. By the above reasoning, $\widetilde{c}_{n_0}(\tau) = \widetilde{c} \, y^{\frac{3}{2} - k}$ for some $0 \ne \widetilde{c} \in \CC$, and $\widetilde{z_0}$ is a polynomial in $\tau$. Observe that $\widetilde{z_0}$ has an analytic continuation to $\HS$, and $\phi$ has singularities along $\widetilde{z_0}(\tau)$ for all $\tau$.
Compare the $n_0$-th Laurent series coefficients of $\cP$ and $\widetilde{\cP}$ at $(\tau_0, z_0) = A_S^{-1} (\widetilde{\tau}_0, \widetilde{z}_0)$:
\begin{gather*}
  \big(\widetilde{c}_{n_0} \big|_{k, m}^\sk A_S^{-1}\big)(\tau) \, \Big(\frac{-z}{\tau} - \widetilde{z}_0 \Big(\frac{-1}{\tau}\Big) \Big)^{n_0}
=
  \widetilde{c}_{n_0}\Big(\frac{-1}{\tau}\Big) |\tau|^{-1} \ov{\tau}^{1 - k} (- \tau)^{-n_0}\, \big(z - z_0(\tau) \big)^{n_0}
\text{.}
\end{gather*}
The fact that $\phi=\phi\,\big|_{k,m}^{sk} A_S$ implies that
\begin{gather*}
  \widetilde{c} \; {\rm Im}\big(\frac{-1}{\tau}\big)^{\frac{3}{2} - k} \, |\tau|^{-1} \ov{\tau}^{1 - k} (-\tau)^{-n_0} 
=
  (-1)^{n_0} \widetilde{c} \; y^{\frac{3}{2} - k} \, \tau^{k-2-n_0} \ov{\tau}^{- 1}
=
  c \, y^{\frac{3}{2} - k}
\text{,}
\end{gather*}
which is impossible, since $c, \widetilde{c} \ne 0$.  This contradiction completes the proof.
\end{proof}

\begin{corollary}
\label{cor:semimeromorphic_skew_maass_jacobi_forms}
We have $\JJskkm = \holJJskkm$, and $\holJJskkm = \{0\}$ if $m < 0$.
\end{corollary}
\begin{proof}
Let $\phi \in \JJskkm$, and assume that $\phi(\tau_0, \,\cdot\,)$ has singularities for some $\tau_0\in\HS$.  Consider the Laurent series expansion of $\phi$ around a singular point $\big(\tau_0, z_0(\tau_0)\big)$ (see Proposition~\ref{prop:laurent_expansion}):
\begin{gather*}
  \sum_{n \ge -N} c_n(\tau) \big(z - z_0(\tau)\big)^n
\end{gather*}
for some $N > 0$ and $c_{-N} \ne 0$.  The functions $c_n :\, \HS \rightarrow \CC$ are real-analytic and $z_0(\tau)$ parametrizes the singularities in a neighborhood of $\big(\tau_0, z_0(\tau_0)\big)$.  However, Proposition~\ref{prop:singular_skew_maass_jacobi_forms} implies the contradiction $c_n = 0$ for $n < 0$.  Hence $\phi$ has no singularities and $\JJskkm = \holJJskkm$.

The second part follows from the residue theorem as in the proof of \cite[Theorem~1.2]{EZ}.
\end{proof}

We are now in a position to prove Theorem~\ref{thm:structure_maass_jacobi_forms}~(\ref{thm:it:semimeromorphic_maass_jacobi_forms}).

\begin{proof}[Proof of Theorem~\ref{thm:structure_maass_jacobi_forms}~(\ref{thm:it:semimeromorphic_maass_jacobi_forms})]
If $\phi \in \JJkm$, then Corollary~\ref{cor:commuting_xi_operators} and Corollary~\ref{cor:semimeromorphic_skew_maass_jacobi_forms} imply that $\xiJkm (\phi) \in \holJsk_{3 - k, m}$. Moreover, if $\xiJkm (\phi) \ne 0$, then $m > 0$.  Note that $\xiJkm :\, \holJJkm \rightarrow \holJsk_{3 -k, m}$ is surjective.  For the subspace of cusp forms of $\holJsk_{3 - k, m}$, this observation is the remark after Theorem~2 of \cite{B-R-Maass-Jacobi}.  It is easy to see that the argument with Jacobi\nbd Poincar\'e series given there holds for all weak skew-holomorphic Jacobi forms of weight $3 - k$ and index $m$.  In particular, there exists $\psi \in \holJJkm$ such that $\xiJkm (\psi) = \xiJkm (\phi)$.  We find that $\phi - \psi$ is meromorphic, and $\phi = \psi + ( \phi - \psi)$ is the desired decomposition.
\end{proof}

We end this section with a corollary, whose proof does not rely on  Proposition~\ref{prop:hharmonic_fourierexpansion}, Proposition~\ref{prop:skew-hharmonic_fourierexpansion}, and Proposition~\ref{prop:images_of_xi}.

\begin{corollary}
\label{cor:singularities_of_maass_jacobi_forms}
Let $\phi \in \JJHkm$.  Then $\partial_{\ov z} (\phi)$ has no singularities.
\end{corollary}
\begin{proof}
If $\partial_{\ov z} (\phi)$ had singularities, then so would $\xiJHkm (\phi)$. However, (\ref{eq:xi_equiv}), (\ref{eq:lapH_factor}), and Corollary~\ref{cor:semimeromorphic_skew_maass_jacobi_forms} yield that $\xiJHkm (\phi)\in\holJJsk_{k,-m}$ has no singularities.
\end{proof}

\section{Theta decompositions}
\label{sec:thetadecompositions}

It is well-known that holomorphic and skew-holomorphic Jacobi forms have a theta decomposition (see \cite{EZ, Sko-Invent90}).  This follows directly from the invariance under the Heisenberg part of $\JacF$, and hence semi-holomorphic forms in $\holJJkm$ and $\holJJskkm$ also have such a theta decomposition.  Specifically, if $0\not = \phi \in \holJJkm$ is semi-holomorphic, then $m > 0$ by Theorem~\ref{thm:structure_maass_jacobi_forms}~(\ref{thm:it:semimeromorphic_maass_jacobi_forms}), and
\begin{gather*}
  \phi(\tau, z)
=
  \sum_{l \pmod{2 m}} h_l(\tau) \, \theta_{m,l}(\tau, z)
\text{,}
\end{gather*}
where $h_l$ are harmonic weak Maass forms and
\begin{gather}
\label{eq:theta_series_def}
  \theta_{m,l}(\tau, z)
:=
  \sum_{r \equiv l \pmod{2 m}} q^{\frac{r^2}{4 m}} \zeta^r
\text{,}
\end{gather}
where we write again $q:=e^{2 \pi i \tau}$ and $\zeta:=e^{2\pi iz}$. For semi-holomorphic skew-Maass-Jacobi forms we have an analogous decomposition $\sum_l \ov{h_l}\, \theta_{m,l}$.

We now review a more precise viewpoint of the theta decomposition.  Recall that the metaplectic cover $\Mp{2}(\ZZ)$ of $\SL{2}(\ZZ)$ is generated by
$
  T
:=
  \big( \left(\begin{smallmatrix}1 & 1 \\ 0 & 1\end{smallmatrix}\right),\, 1 \big)
$ and
$
  S
:=
  \big( \left(\begin{smallmatrix}0 & -1 \\ 1 & 0\end{smallmatrix}\right),\, \sqrt{\tau} \big)
$, where the root is given by the principal branch.  The Weil representation $\rho_m$ of $\Mp{2}(\ZZ)$ associated to the Jacobi index $m > 0$ is defined as follows (for example, see~\cite{Sko-Weil} for details).  It is a representation of $\Mp{2}(\ZZ)$ on the group algebra $\CC \big[ \ZZ / 2 m \ZZ \big]$, which has canonical basis elements $\frake_l$ for $l \in \ZZ / 2 m \ZZ$:
\begin{align}
  \rho_m (T) \, \frake_l
&:=
  e_{4 m}(l^2) \, \frake_l
\text{,}
\\[6pt]
  \rho_m (S) \, \frake_l
&:=
  \frac{1}{\sqrt{2 i m}}
  \sum_{l' \pmod{2 m}} \hspace{-1em}
  e_{2m}(- l l') \frake_{l'}
\text{,}
\end{align}
where here and throughout this section, $e_m(w):=e^{\frac{2 \pi i w}{m}}$.  We denote the dual Weil representation by ${\check \rho}_{m}$.

The Weil representation factors over the congruence subgroup
\begin{gather*}
  \Mp{2}(\ZZ)[4m]
:=
  \Big\{ \left(\begin{matrix} a & b \\ c & d \end{matrix}\right) \,:\, a \equiv d \equiv 1 \pmod{4m} \text{ and } b \equiv c \equiv 0 \pmod{4m} \Big\}
\text{.}
\end{gather*}

Given $h \,:\, \HS \rightarrow \CC[\ZZ / 2 m \ZZ]$, we define a vector-valued slash action of $\Mp{2}(\ZZ)$:
\begin{gather*}
  h \big|_{k, \rho_m} \, g
:=
  \rho_m(g) \, h \big|_k \, g 
\text{}
\end{gather*}
for all $g\in\Mp{2}(\ZZ)$. We say that a map $h \,:\, \HS \rightarrow \CC[\ZZ / 2 m \ZZ]$ is a vector-valued modular form if every component is a modular form (for some congruence subgroup) and if $h$ is invariant under the $\big|_{k, \rho_m}$-action of $\Mp{2}(\ZZ)$.  This definition extends to vector-valued harmonic weak Maass forms of weight $k$ and type $\rho_m$.  We write $\bbM_{k, \rho_m}$ for the space of such forms, and ${\rm M}^!_{k, \rho_m} \subset \bbM_{k, \rho_m}$ for the subspace of weakly holomorphic vector-valued modular forms of weight $k$ and type $\rho_m$. Vector-valued Jacobi forms can be defined analogously.  The transformation laws of $\theta_{m,l}$ (see $\S 5$ of \cite{EZ}) yield that $(\theta_{m,l})_l$ is a vector-valued Jacobi form of weight $\frac{1}{2}$, index $m$, and type ${\check \rho}_m$.


The theta decomposition for Jacobi forms can be stated more precisely as an isomorphism between vector-valued modular forms and Jacobi forms (for example, see~\cite{Sko-Weil}).  It is easy to see that such isomorphisms hold also for semi-holomorphic forms in $\holJJkm$ and $\holJJskkm$.  Specifically,
\begin{alignat}{2}
\label{eq:theta_expansion_map}
  \bbM_{k - \frac{1}{2},{\rho}_m} &\longrightarrow \holJJkm, \;
&\quad
  (h_l)_l &\longmapsto \sum_l h_l \, \theta_{m,l}
\qquad\text{and}
\\[6pt]
\label{eq:theta_expansion_map_skew}
  \bbM_{k - \frac{1}{2},\check \rho_m} &\longrightarrow \holJJskkm, \;
&\quad
  (h_l)_l &\longmapsto \sum_l \ov{h_l} \, \theta_{m,l}
\end{alignat}
are bijective for $m > 0$.

We next recall a set of $\mu$-functions from Zwegers \cite{Zwe-thesis, Zwe-multivar-Appell} that will serve as a substitute for the theta series in (\ref{eq:theta_series_def}).  
Let $m > 0$.  For $n \in \ZZ^{2m}$, write $|n| := \sum_{i = 1}^{2 m} n_i$ and $\|n\| := \sum_{i = 1}^{2 m} n_i^2$.  Define
\begin{gather}
\mu_m(z_1, z_2; \tau):=
  \frac{e^{\pi i z_1}}{\theta(z_2; \tau)^{2 m}}
  \sum_{n \in \ZZ^{2 m}} \frac{(-1)^{|n|} q^{\frac{1}{2}\|n\|^2 + \frac{1}{2}|n|} e^{2\pi i|n|z_2}}
                          {1 - e^{2\pi iz_1} q^{|n|}},
\end{gather}
where 
\begin{gather}
  \theta(z; \tau)
:=
  \sum_{r \in \ZZ + \frac{1}{2}} (-1)^{r + \frac{1}{2}} q^{\frac{r^2}{2}} \zeta^r
\end{gather}
is a Jacobi theta function, and
\begin{gather*}
  R(z;\tau)
:=
  \sum_{n \in \ZZ + \frac{1}{2}} \!
  \Big(\, \sgn(n) - E\Big( \sqrt{2 y} \big(n + \tfrac{v}{y}\big) \Big)\, \Big)
  (-1)^{n - \frac{1}{2}} q^{-\frac{n^2}{2}} \zeta^{-n}
\text{,}
\end{gather*}
where
\begin{gather*}
  E(w)
:=
  2 \int_0^w \! e^{- \pi u^2} \;du
=
  \tfrac{\sgn(w)}{\sqrt{\pi}} \gamma\big( \tfrac{1}{2}, \pi w^2 \big)
\end{gather*}
is the error function. Set
\begin{align}
\label{eq:mu_hat_ml_definition}
  {\widehat \mu}_{m,l} (z; \tau)
& :=
  (-1)^m q^{\frac{-(l + m)^2}{4m}} \zeta^{-(l + m)}
\\[4pt]
\nonumber
& \qquad
  \cdot
  \Big( \mu_m\big(\tfrac{1}{2} + (l + m) \tau, \tfrac{1}{4m} - z; \tau\big)
        - \tfrac{i}{2} R\big( 2 m z + (l + m) \tau - \tfrac{2m + 1}{2}; 2 m \tau \big) \Big)
\text{.}
\end{align}


Note that ${\widehat \mu}_{m,l} (z; \tau)$ in (\ref{eq:mu_hat_ml_definition}) coincides with $(-1)^l\, {\widehat \mu}_{2 m, l + m}(u, v; \tau)$ of~\cite{Zwe-multivar-Appell} evaluated at $u = \frac{1}{2}$ and $v = \big( \frac{1}{4m} - z, \ldots, \frac{1}{4m} - z \big)$, and \cite[Theorem~4.5]{Zwe-multivar-Appell} immediately implies:


\begin{proposition}
\label{prop:vector_valued_mu_function}
The vector $(\widehat{\mu}_{m,l})_l$ is a vector-valued Jacobi form of weight $\frac{1}{2}$, index $-m$, and of type ${\rho}_{m}$.  More precisely,
\begin{align*}
  \widehat{\mu}_{m,l} \big|_{\frac{1}{2}, -m} \left[\big(\left(\begin{smallmatrix}1 & 1\\0 & 1\end{smallmatrix}\right), \sqrt{1} \,\big), (0, 0)\right]
& =
  e_{4 m} (-l^2)\, \widehat{\mu}_{m,l}
\qquad\text{and}
\\[6pt]
  \widehat{\mu}_{m,l} \big|_{\frac{1}{2}, -m} \left[\big(\left(\begin{smallmatrix}0 & -1\\1 & 0\end{smallmatrix}\right), \sqrt{\tau} \,\big), (0, 0)\right]
& =
  \frac{i}{\sqrt{2 i m}}
  \sum_{l' \pmod{2 m}} \hspace{-1em} e_{2 m}(l l')\, \widehat{\mu}_{m,l'}
\text{.}
\end{align*}
\end{proposition}

The following theorem is one of our main results, which provides a theta-like decomposition for \Hharmonic\ Maass-Jacobi forms.

\begin{theorem}
\label{thm:hharmonic_maass_jacobi_forms_thetadecomposition}
Let $m > 0$.  The map
\begin{align}
\label{eq:mu_theta_expansion_map}
  {\rm M}^!_{k - \frac{1}{2},\check \rho_{m}} \times \J_{k,-m} &\longrightarrow \JH_{k,-m},
\\[6pt]
\nonumber
  \big( (h_l)_{l \pmod{2 m}},\, \varphi \big) &\longmapsto \sum_{l \pmod{2 m}} h_l \, \widehat{\mu}_{m,l} + \varphi
\end{align}
is bijective.
\end{theorem}

\begin{remark}
\ 
\begin{enumerate}[(1)]
\item 
If the ``meromorphic part'' $\varphi$ in Theorem~\ref{thm:hharmonic_maass_jacobi_forms_thetadecomposition} has poles only at torsion points, then it has a decomposition into a so-called polar part and a finite part, which admits a theta decomposition involving mock modular forms (for details see~\cite{Zwe-thesis} and~\cite{DMZ}).

\item
Note that Theorem \ref{thm:intro_thetadecomposition} is simply a reformulation of Theorem~\ref{thm:hharmonic_maass_jacobi_forms_thetadecomposition}.  We find that
\begin{gather*}
  \xiJHkm :\, \JHkm \slashdiv \Jkm \tilde\longrightarrow \holJsk_{k,-m}
\end{gather*}
is Hecke equivariant.  In particular, skew-holomorphic Jacobi Hecke eigenforms correspond to ``Hecke eigenforms'' in the subspace of moderate growth \Hharmonic\ Maass-Jacobi forms.
\end{enumerate}
\end{remark}
\begin{proof}[Proof of Theorem~\ref{thm:hharmonic_maass_jacobi_forms_thetadecomposition}]
The map~\eqref{eq:mu_theta_expansion_map} is well-defined by Proposition~\ref{prop:vector_valued_mu_function}.  Note that we have not yet used the growth condition~(\ref{it:growth_condition}) in Definition~\ref{def:maassjacobiforms}.  As a first step, we will establish a weaker version of Theorem~\ref{thm:hharmonic_maass_jacobi_forms_thetadecomposition}, where the growth conditions of the left and right hand sides of~\eqref{eq:mu_theta_expansion_map} are removed.  We will denote this weaker map by \eqref{eq:mu_theta_expansion_map}'.  The theorem then follows from Proposition~\ref{prop:singularities_of_maass_jacobi_forms}, whose proof only relies on the weaker version of Theorem~\ref{thm:hharmonic_maass_jacobi_forms_thetadecomposition}.  For the remainder of this proof we implicitly remove the growth condition for all spaces of modular forms and Jacobi forms that occur.

A direct computation shows that
\begin{gather*}
\xiJH_{\frac{1}{2}, -m} \big( {\widehat \mu}_{m,l}\big)=\theta_{m,l}
\text{,}
\end{gather*}
and the linear independence of $z \mapsto \xiJH_{\frac{1}{2}, -m} \big( {\widehat \mu}_{m,l}(z; \tau) \big)=\theta_{m,l}(\tau,z)$, for $l=1,\ldots,2m$, and for any fixed $\tau$ establishes the injectivity of~\eqref{eq:mu_theta_expansion_map}'.  It remains to prove that~\eqref{eq:mu_theta_expansion_map}' is surjective.  Let $\phi \in \JH_{k,-m}$.  Equations (\ref{eq:xi_equiv}) and (\ref{eq:lapH_factor}) and Corollary~\ref{cor:semimeromorphic_skew_maass_jacobi_forms} imply that $\xiJH_{k,-m} (\phi)\in\holJsk_{k,m}$.  In particular, $\xiJH_{k,-m} (\phi)$ has a theta decomposition of the form $\sum_l \ov{h_l} \, \theta_{m,l}$ (see \eqref{eq:theta_expansion_map}) and $\psi := \sum h_l \, {\widehat \mu}_{m,l} \in \JH_{k,-m}$ by Proposition~\ref{prop:vector_valued_mu_function}.  We have $\xiJH_{k,-m}(\psi) = \xiJH_{k,-m}(\phi)$, so that $\varphi:=\phi - \psi \in \J_{k,-m}$, which yields the surjectivity of~\eqref{eq:mu_theta_expansion_map}'.
\end{proof}

We now prove Theorem~\ref{thm:structure_maass_jacobi_forms}~(\ref{thm:it:singularity_free_hharmonic}), (\ref{thm:it:decomposition_of_maass_jacobi_forms}), and~(\ref{thm:it:positive_index_implies_semi_meromorphic}), where we will repeatedly employ the following fact already used in the proof of Corollary~\ref{cor:semimeromorphic_skew_maass_jacobi_forms}:  If a non-zero semi-holomorphic function $\phi$ satisfies the elliptic transformation property of a Jacobi form of index $m$  (i.e., $\phi$ is invariant under $|_{k,m}\left[\big(\left(\begin{smallmatrix}1 & 0\\0 & 1\end{smallmatrix}\right), \sqrt{1} \,\big), (\lambda, \mu)\right]$ for $\lambda, \mu \in \ZZ$), then $m>0$.  This follows exactly as in the proof of \cite[Theorem~1.2]{EZ}.

\begin{proof}[Proof of Theorem~\ref{thm:structure_maass_jacobi_forms}~(\ref{thm:it:singularity_free_hharmonic})]
If $m < 0$, then $\holJJkm = \holJkm = \{ 0 \}$ by the above fact.  The second equality in the theorem follows from the first, because $\holJHkm = \holJJHkm \cap \ker(\xiJkm)$ and $\holJkm = \holJJkm \cap \ker(\xiJkm)$.  We now show that $\holJJHkm = \holJJkm$. Suppose that $\phi \in \holJJHkm$, but $\phi \not\in \holJJkm$.  Then (\ref{eq:xi_equiv}), (\ref{eq:lapH_factor}), and Corollary~\ref{cor:semimeromorphic_skew_maass_jacobi_forms} imply that $0 \ne \xiJHkm(\phi) \in \holJJsk_{k, -m}$ is semi-holomorphic.  Hence $-m > 0$ and $\xiJHkm(\phi)$ has a theta decomposition of the form $\sum_l \ov{h_l} \, \theta_{m,l}$. We use the same idea as in the proof of Theorem~\ref{thm:hharmonic_maass_jacobi_forms_thetadecomposition}. Consider $\psi(\tau, z) = \sum h_l(\tau) \, {\widehat R}_{m,l}(z; \tau)$, where 
\begin{gather*}
 {\widehat R}_{m,l} (z; \tau)
:=
  (-1)^{m + 1} \tfrac{i}{2} q^{\frac{-(l + m)^2}{4m}} \zeta^{-(l + m)} \;
  R\big( 2 m z + (l + m) \tau - \tfrac{2m + 1}{2}; 2 m \tau \big)
\end{gather*}
is the ``non-holomorphic'' part of (\ref{eq:mu_hat_ml_definition}).  Then $\psi$ (not modular in~$\tau$) has no singularities, and $0 \ne \phi - \psi$ is semi-holomorphic and elliptic in $z$ (see~\cite{Zwe-multivar-Appell}). Thus, $m > 0$.  This contradiction completes the proof.
 



\end{proof}

\begin{proof}[Proof of Theorem~\ref{thm:structure_maass_jacobi_forms}~(\ref{thm:it:decomposition_of_maass_jacobi_forms}) ] Let $\phi \in \JJHkm$.  Then $\xiJHkm (\phi)\in\holJJsk_{k,-m}$ by (\ref{eq:xi_equiv}), (\ref{eq:lapH_factor}),  and Corollary~\ref{cor:semimeromorphic_skew_maass_jacobi_forms}.  In particular, all principal parts of $\phi$ are semi-meromorphic, and hence the same is true for $\xiJkm (\phi)$.  Corollary~\ref{cor:commuting_xi_operators}  and Proposition~\ref{prop:singular_skew_maass_jacobi_forms} imply that $\xiJkm (\phi) \in \holJHsk_{3-k,m}$.  Now, if $\xiJkm (\phi)$ were not annihilated by $\xiJHsk_{3 - k, m}$, then $m < 0$, since $\xiJHsk_{3 - k, m} \big( \xiJkm (\phi) \big)\in\holJ_{3-k,-m}$ is semi-holomorphic. As in the proof of Theorem~\ref{thm:structure_maass_jacobi_forms}~(\ref{thm:it:singularity_free_hharmonic}), we find some $\psi(\tau, z) = \sum h_l(\tau) \, \widehat{R}_{m,l}(z; \tau)$ (not modular in~$\tau$) without singularities such that $0 \ne \xiJkm (\phi) - \psi$ is semi-holomorphic and elliptic in $z$.  Then $m > 0$, which is a contradiction to our previous finding.

Thus, $\xiJkm (\phi) \in \holJsk_{3-k,m}$.  Recall from the proof of 
Theorem~\ref{thm:structure_maass_jacobi_forms}~(\ref{thm:it:semimeromorphic_maass_jacobi_forms}) that $\xiJkm :\, \holJJkm \rightarrow \holJsk_{3 -k, m}$ is surjective.  Hence there exists a $\widetilde{\phi}\in\holJJkm $ such that $\phi - \widetilde{\phi}$ vanishes under $\xiJkm$, which establishes the claim.
\end{proof}

\begin{proof}[Proof of Theorem~\ref{thm:structure_maass_jacobi_forms}~(\ref{thm:it:positive_index_implies_semi_meromorphic})]
If $\phi \in \JJHkm$, then $\xiJHkm (\phi)\in\holJJsk_{k,-m}$ by (\ref{eq:xi_equiv}), (\ref{eq:lapH_factor}),  and Corollary~\ref{cor:semimeromorphic_skew_maass_jacobi_forms}. Moreover, if $\xiJHkm \big(\phi\big) \ne 0$, then Corollary~\ref{cor:semimeromorphic_skew_maass_jacobi_forms} asserts that $-m > 0$, yielding the first equality. The second equality follows from the first, since $\JHkm = \JJHkm \cap \ker(\xiJkm)$ and $\Jkm = \JJkm \cap \ker(\xiJkm)$.
\end{proof}

We have now settled all analytic and structural properties of \Hharmonic\ Maass Jacobi forms.  We emphasize that we have not yet used the growth condition~(\ref{it:growth_condition}) of Definition~\ref{def:maassjacobiforms}. To complete the proof of Theorem~\ref{thm:hharmonic_maass_jacobi_forms_thetadecomposition}, we have to show that the growth condition~(\ref{it:growth_condition}) of Definition~\ref{def:maassjacobiforms} implies the growth condition for harmonic weak Maass forms on the left hand side of~\eqref{eq:mu_theta_expansion_map}.
\begin{proposition}
\label{prop:singularities_of_maass_jacobi_forms}
Fix $\phi \in \JJHkm$.  Then for all but finitely many $\alpha, \beta \in \QQ \pmod{\ZZ}$, the set $\{(\tau, \alpha\tau + \beta) \,:\, \tau \in \HS\}$ is not a polar divisor of $\phi$.  For every such $\alpha, \beta$, the function $\phi(\tau, \alpha \tau + \beta)$ has no singularities for sufficiently large~$y$.
\end{proposition}
\begin{proof}
Fix $\tau \in \HS$.  By Corollary~\ref{cor:singularities_of_maass_jacobi_forms}, the set of singularities of $\phi(\tau, \,\cdot\,)$ is discrete in~$\CC$.  In particular, there are at most finitely many $\alpha, \beta \in \QQ \pmod{\ZZ}$ such that $\phi(\tau, \alpha \tau + \beta)$ is a pole.  This proves the first part.

To establish the second part, it suffices to show the claim for $\phi \in \JHkm$, since $\JJHkm = \holJJkm + \JHkm$ by Theorem~\ref{thm:structure_maass_jacobi_forms}.  We employ the map~\eqref{eq:mu_theta_expansion_map}' defined in the proof of Theorem~\ref{thm:hharmonic_maass_jacobi_forms_thetadecomposition}.  Write $\phi$ as $\sum_{l \pmod{2 m}} h_l\, \widehat{\mu}_{m, l} + \psi$.  Note that if $\alpha, \beta \in \QQ$ such that $\widehat{\mu}_{m, l}(\tau,\, \alpha\tau + \beta)$ is defined, then it has no singularities.  Hence it remains to consider the meromorphic Jacobi form~$\psi$.  Now, since $\psi$ is meromorphic, $\psi(\tau, \alpha \tau + \beta)$ is meromorphic, too.  This implies that for sufficiently large $y$, it has no singularities, proving the proposition.
\end{proof}

We conclude the section with a remark.
\begin{remark}
In Example~\ref{ex:maass_jacobi_spaces}~(\ref{ex:maass_jacobi_spaces:it:mu_function}), we pointed out that Zwegers's \cite{Zwe-thesis} $\widehat{\mu}$-function has a decomposition of the form $\widehat{\mu} = \mu_1 + \widehat{\mu}_2$, where $\widehat{\mu}_2 \in \JH_{\frac{1}{2}, -\frac{1}{2}}$.  Such a decomposition can for example be found by setting
\begin{gather}
\label{eq:mu_onehalf_def}
  \widehat{\mu}_2 (z; \tau)
:=
  \widehat{\mu} \big( z + \tfrac{1 + \tau}{2}, \tfrac{1 + \tau}{2}; \tau \big)
\text{.}
\end{gather}
Up to meromorphic Jacobi forms, $\widehat{\mu}_{2}$ is essentially the only Jacobi form that can be obtained as a ``specialization'' of $\widehat{\mu}$ (see~\cite{Z-Bourbaki}).
Moreover, there is no meromorphic Jacobi form $h$ such that $\widehat{\mu}_2 + h$ has no singularities.  One can see this by considering the residues of the poles of $z \mapsto \widehat{\mu}_2(z; \tau)$.  More precisely, suppose that $g$ is a meromorphic Jacobi form of index~$0$ such that the Jacobi form (on $\Mp{2}(\ZZ) \ltimes (2 \ZZ)^2$)
\begin{gather}
\label{eq:mu_potential_correction}
  \widehat{\mu}_{2} (z; \tau)
- \frac{g(\tau, z)}
       {e^{\pi i z} \theta \big(\tau, z + \frac{1 + \tau}{2}\big)}
\end{gather}
has no singularities. Then $g$ is holomorphic, since the zeros of the denominator of the second term in (\ref{eq:mu_potential_correction}) occur precisely where $\widehat{\mu}_2$ has simple poles.  Thus, $g$ is independent of $z$, i.e., $g$ is a weakly holomorphic modular form.  Suppose that the residues at $\frac{-1 + \tau}{2}$ of the first and second term in (\ref{eq:mu_potential_correction}) are the same.  Then the transformation behavior of $\widehat{\mu}$ and $\theta$ under $z \mapsto z + 1$ shows that the residues of these terms at $\frac{1 + \tau}{2}$ differ by a sign.  In particular, the residues will not cancel, and hence there is no $g$ such that (\ref{eq:mu_potential_correction}) has no singularities.
\end{remark}


\section{H-quasi Maass-Jacobi forms}
\label{sec:quasijacobiforms}

Kaneko and Zagier~\cite{Kan-Z} introduced the space of quasimodular forms, which includes the Eisenstein series~$E_2$.  Quasimodular forms impact various aspects of automorphic forms and physics, and the theory has been extended to the setting of Jacobi forms~(for example,  see \cite{Kaw-Yos, Lib-MSRI11}).  The notion of quasi-Jacobi forms in the literature mimics the definition of quasimodular forms by Kaneko and Zagier somewhat closely, and ``quasimodular behavior'' with respect to the Jacobi variable $z$ has not been considered thus far.  In this section, we fill this gap by introducing completed \Hquasi\ Maass-Jacobi forms (see Definition~\ref{def:quasi_maass_jacobi_forms}).  Note that examples of such forms have recently appeared as generating functions of Gromov Witten invariants in~\cite{Ober}.  The main result of this section (Theorem~\ref{thm:Hquasi_structure}) gives a characterization of completed \Hquasi\ Maass-Jacobi forms in terms of  \Hharmonic\ Maass-Jacobi forms, which implies that there exists no Jacobi form analog of the quasimodular Eisenstein series~$E_2$.

With an abuse of notation we suppress from now on the superscripts and simply write $X_\pm$ and $Y_\pm$ for the operators defined in Section~\ref{sec:differentialoperators}.  Recall that every quasimodular form can be completed to a (real-analytic) modular form $f:=\sum_{d = 0}^{D-1} y^d f_d$ with holomorphic $f_d$. Then $f$ is annihilated by  $X_-^D$.  More generally, if the functions $f_d$ are only harmonic, then $f$ is annihilated by $X_+^D X_-^D$.  This motivates the next definition of completed \Hquasi\ Maass-Jacobi forms, where as before $m \ne 0$.

\begin{definition}
\label{def:quasi_maass_jacobi_forms}
Let $\phi :\, \HS \times \CC \rightarrow \CC$ be a real-analytic function except for possible singularities of type $f g^{-1}$, where $f$ and $g$ are real-analytic, such that the singularities of $\phi(\tau, \,\cdot\,)$ are isolated for every $\tau \in \HS$.  Then $\phi$ is a completed \Hquasi\ Maass-Jacobi form of weight $k$, index $m$, and depth~$D$ if the following conditions are satisfied:
\begin{enumerate}[(1)]
\item For all $A \in \JacF$, we have $\phi \big|_{k, m} A = \phi$.
\vspace{1ex}
\item We have that $\cC_{k,m}(\phi)=0$.
\vspace{1ex}
\item \label{it:quasi_harmonicity_condition} We have that $Y_+^D Y_-^D \big(\phi\big) = 0$.
\vspace{1ex}
\item For every $\alpha, \beta \in \QQ$ such that $\{(\tau, \alpha\tau + \beta) \,:\, \tau \in \HS\}$ is not a polar divisor of $\phi$, we have that $\phi (\tau, \alpha\tau+\beta) = O\big( e^{a y} \big)$ as $y \rightarrow \infty$ for some $a>0$.
\end{enumerate}
\end{definition}

\begin{remark}
\label{rm:hquasi_definition}
\
\begin{enumerate}[(1)]
\item  \Hquasi\ Maass-Jacobi forms of depth~$D=1$ are \Hharmonic\ Maass-Jacobi forms.
\item 
\label{rm:skew_quasi} One can define completed \Hquasi\ skew-Maass-Jacobi forms by replacing $\big|_{k,m}$, $\cC_{k,m}$, and  $Y_{\pm}$ in Definition~\ref{def:quasi_maass_jacobi_forms} with their skew-analogs.
\item  Observe that the commutator $[Y_-, Y_+]=- 2 \pi m$.  Hence the operator $Y_+^D Y_-^D$ can be expressed as a polynomial in the Heisenberg Laplace operator:
\begin{gather}
\label{eq:quasi_differential_operator_factorization}
  Y_+^D Y_-^D
=
  \prod_{d = 0}^{D-1} \big( \lapH_m + 2 \pi m d \big)
\text{.}
\end{gather}
Analogously, in the quasimodular setting $X_+^D X_-^D$ can be expressed as a polynomial in the hyperbolic Laplace operator:
\begin{gather*}
  X_+^D X_-^D
=
  \prod_{d = 0}^{D - 1} \big( \Delta_k + (k - 2 d)_d \big)
\text{,}
\end{gather*}
where $(n)_l := \prod_{i = 0}^{l - 1} (n - i)$ is the Pochhammer symbol ($(n)_0 := 0$).
\end{enumerate}
\end{remark}

The following two results give descriptions of completed \Hquasi\ Maass-Jacobi forms.
\begin{lemma}
\label{lm:semiholomorphic_quasi_maass_jacobi_forms}
Let $\phi$ be a completed \Hquasi\ Maass-Jacobi form of weight~$k$, index~$m$, and depth~$D$ that is annihilated by $Y_-^D$.  Then
\begin{gather*}
  \phi
=
  \sum_{d = 0}^{D - 1} Y_+^d \big( \phi_d \big)
\text{,}
\end{gather*}
where $\phi_d \in\JJ_{k - d, m} $.
\end{lemma}
\begin{proof}
We induct on~$D$.  The case~$D = 1$ is clear by Definition~\ref{def:maassjacobiforms}.  Let $D > 1$, and set $\phi_D := Y_-^{D - 1} (\phi)$.  Then  $\phi_D \in\JJ_{k +1- D, m} $. Consider $\widetilde{\phi} := \phi - \big( (-2 \pi m)^{D - 1} (D - 1)! \big)^{-1} Y_+^{D-1} \big( \phi_D \big)$. With the help of (\ref{eq:quasi_differential_operator_factorization}) and the fact that $[Y_-,\lapH_m]=- 2 \pi mY_-$, we verify that
\begin{align*}
  Y_-^{D - 1} \, \big( \widetilde{\phi} \big)
&=
 Y_-^{D-1} (\phi)
  - \frac{1}{(-2 \pi m)^{D - 1} (D - 1)!} \;
    Y_-^{D-1} \Big( \prod_{d = 0}^{D - 2} \big( \lapH_m + 2 \pi m d \big) \Big) \big( \phi \big)
\\[4pt]
&=
 Y_-^{D-1} (\phi)
  - \frac{1}{(-2 \pi m)^{D - 1} (D - 1)!} \;
    \Big( \prod_{d = 0}^{D - 2} \big( \lapH_m + 2 \pi m d - 2 \pi m (D - 1) \big) \Big) Y_-^{D-1} \big( \phi \big)
\\[4pt]
&=
 \phi_D
  - \frac{1}{(-2 \pi m)^{D - 1} (D - 1)!} \;
    \prod_{d = 1}^{D - 1} \big( -2 \pi m d \big) \big( \phi_{D} \big)
=
  0
\text{.}
\end{align*}
Thus, $\widetilde{\phi}$  is a completed \Hquasi\ Maass-Jacobi form of weight~$k$, index~$m$, and depth~$D - 1$, and the claim follows by induction.
\end{proof}

\begin{theorem}
\label{thm:Hquasi_structure}
Let $\phi$ be a completed \Hquasi\ Maass-Jacobi form of weight $k$, index $m$, and depth~$D$.  Then
\begin{gather*}
  \phi
=
  \sum_{d = 0}^{D - 1} Y_+^d \big( \phi_d \big)
\text{,}
\end{gather*}
where $\phi_d \in \JJH_{k - d, m} $.
\end{theorem}
\begin{proof}
First note that for any $\phi$ and $m < 0$, we have
\begin{gather}
\label{eq:intertwining_skew_raising}
  \sqrt{-m y}^{\, -1} \exp\big(- 4 \pi m \tfrac{v^2}{y} \big) \; \ov{Y_+ (\phi)}
= 
  Y^\sk_- \Big( \sqrt{-m y}^{-1} \exp\big(- 4 \pi m \tfrac{v^2}{y} \big) \, \ov{\phi} \Big)
\text{.}
\end{gather}
Similar relations hold also for $Y_-$ and $Y_+^\sk$.

Now, set $\widetilde{\phi} :=  \sqrt{-my}^{-1} \exp\big(-4 \pi m \tfrac{v^2}{y}\big)\, \ov{Y_-^D (\phi)}$, and assume that $\widetilde{\phi} \ne 0$. Then $\widetilde{\phi}$ is a completed \Hquasi\ skew-Maass-Jacobi form of weight $k+1-D$,  index $-m$, and depth $D$ (see Remarks~\ref{rm:hquasi_definition}~(\ref{rm:skew_quasi})).  Using Relation~(\ref{eq:intertwining_skew_raising}), we confirm that $\widetilde{\phi}$ vanishes under $\big(Y_-^\sk\big)^D$.  It is easy to extend Lemma~\ref{lm:semiholomorphic_quasi_maass_jacobi_forms} to completed \Hquasi\ skew-Maass-Jacobi forms, and we find that
\begin{gather*}
  \widetilde{\phi}
=
  \sum_{d = 0}^{D - 1} \big( Y_+^\sk \big)^d \big( \phi_d \big)
\text{,}
\end{gather*}
for some $\phi_d \in\JJsk_{k +1 - D + d, -m}$. 
Proposition~\ref{prop:singular_skew_maass_jacobi_forms} implies that $m < 0$.  Let $\phi^{[\mu]}_{d} \in \JJH_{k + 1 - D + d, m}$ denote the preimage of $\phi_d$ under $\xiJH_{k + 1 - D + d,m}$.  If
\begin{gather*}
  \phi^{[\mu]}
:=
  \sum_{d = 0}^{D - 1} \frac{1}{(- 2 \pi m)^{D - d - 1} (D - d - 1)!} \, Y_+^{D - d - 1} \big( \phi^{[\mu]}_d \big)
\text{,}
\end{gather*}
then $\phi - \phi^{[\mu]}$ vanishes under $Y_-^D$.  Indeed, the image of $\phi^{[\mu]}$ under $Y_-^D$ is given by
\begin{gather*}
  \sum_{d = 0}^{D - 1} \frac{1}{(- 2 \pi m)^{D - d - 1} (D - d - 1)!} \;
    Y_{-}^{d + 1} \prod_{d' = 1}^{D - d - 1} \big( \lapH_m - 2 \pi m d' \big) \big( \phi_d^{[\mu]} \big)
=
  \sum_{d = 0}^{D - 1}Y_{-}^{d + 1} \big( \phi_d^{[\mu]} \big)
\text{.}
\end{gather*}

We obtain
\begin{align*}
  \big( Y_- \big)^D (\phi)
&=
  \ov{ \sqrt{-my}\exp\big( 4 \pi m \tfrac{v^2}{y} \big)\, \widetilde{\phi} }
=
  \sqrt{-my}\exp\big( 4 \pi m \tfrac{v^2}{y} \big) \,
   \ov{ \sum_{d = 0}^{D - 1} \big( Y_+^\sk \big)^d \big( \phi_d \big) }
\\[4pt]
&=
  \sum_{d = 0}^{D - 1} Y_-^d \Big( \sqrt{-my}\exp\big( 4 \pi m \tfrac{v^2}{y} \big) \ov{ \phi_d } \Big)
=
  \sum_{d = 0}^{D - 1} Y_-^d \Big( Y_- \big(  \phi_d^{[\mu]} \big) \Big)
\text{,}
\end{align*}
which proves the theorem after applying Lemma~\ref{lm:semiholomorphic_quasi_maass_jacobi_forms}.
\end{proof}
As an immediate consequence of Theorem  \ref{thm:Hquasi_structure} we record:
\begin{corollary}
The space of all completed \Hquasi\ Maass-Jacobi forms of weight $k$ and index $m$ equals
\begin{gather*}
  \bigoplus_{d = 0}^\infty Y_+^d \, \Big( \JJH_{k - d, m} \Big)
\text{.}
\end{gather*}
\end{corollary}
We end with a final remark.
\begin{remark}
In~\cite{B-F-KacWakimoto}, Folsom and the first author describe a modular completion of characters of $s\ell (m|n)\widehat{\;}$ highest weight modules.  They encounter products of automorphic forms that are in the spirit of Theorem~\ref{thm:Hquasi_structure}.
\end{remark}

{\tit Acknowledgement:} The authors thank Dan Bump and Sander Zwegers for many valuable suggestions.

\bibliographystyle{plain}
\bibliography{Lit}

\end{document}